\documentclass{amsproc}

\usepackage{amssymb}
\usepackage{tikz}

\usepackage{graphicx}
\usepackage{amsmath}
\usepackage{subfigure}
\usepackage{calc}
\usepackage{float}

\newtheorem{theorem}{Theorem}[section]
\newtheorem{lemma}{Lemma}[section]
\newtheorem{proposition}{Proposition}[section]
\newtheorem{remark}{Remark}[section]
\newtheorem{corollary}{Corollary}[section]
\newtheorem{definition}{Definition}[section]

\def\Z{\mathbb{Z}}
\def\R{\mathbb{R}}
\def\C{\mathbb{C}}
\def\H{\mathbb{H}}
\def\N{\mathbb{N}}
\def\F{\mathcal{F}}
\def\L{\mathrm{L}}

\def\0{{\bf{0}}}

\def\vx{{\vec{\xi}}}
\def\vt{{\vec{t}}}
\def\vs{{\vec{s}}}
\def\vr{{\vec{r}}}

\def\-{{\mbox{\tiny $ - $ }}}
\def\+{{\mbox{\tiny $ + $ }}}

\def\0{{\mbox{\tiny $ (0) $ }}}
\def\1{{\mbox{\tiny $ (1) $ }}}
\def\2{{\mbox{\tiny $ (2) $ }}}

\newcommand{\dotcup}{\ensuremath{\mathaccent\cdot\cup}}

\numberwithin{equation}{section}

\usepackage{lipsum}

\newcommand\blfootnote[1]{
  \begingroup
  \renewcommand\thefootnote{}\footnote{#1}%
  \addtocounter{footnote}{-1}%
  \endgroup
}

\begin{document}

\allowdisplaybreaks

   \begin{center}
      \Large\textbf{Equations For Frame Wavelets In $L^2(\R^2)$}
   \end{center}

\vspace{.5cm}

\begin{center}
      \large\textit{Xingde Dai}

      \vspace{.5cm}

      \textit{Dedicated to Zurui Guo }

\end{center}

\bigskip

\begin{center}Abstract\end{center}

\begin{quotation}

{\small We establish system of equations   for single function normalized tight frame wavelets
with compact supports
associated with $2\times 2$ expansive integral matrices in $L^2(\R^2)$.}
\end{quotation}

      \address{The University of North Carolina at Charlotte}
      \email{xdai@uncc.edu}

\blfootnote{2010 Mathematics Subject  Classification. Primary 46N99, 47N99, 46E99; Secondary 42C40, 65T60.}

\section{Introduction}

\bigskip

In this paper $\R^2$ will be the two dimensional  Euclidean space and
and $\C^2$  will be the two dimensional complex  Euclidean space.
We will use notations $\vt,\vec{s},\vr,\vx,\vec{\eta}$ for vectors in $\R^2$ or
$\C^2.$
We will use notation $\vt_1\circ \vt_2$ for the standard inner product of two vectors $\vt_1$ and $\vt_2.$
For a vector $\vx=\left(\begin{array}{c} \xi_1 \\  \xi_2  \end{array}\right)$ in $\C^2,$
its real part is
$\mathfrak{Re}(\vx)\equiv \left(\begin{array}{c} \mathfrak{Re}(\xi_1) \\  \mathfrak{Re}(\xi_2)  \end{array}\right)$
and its imaginary part is
$\mathfrak{Im}(\vx)\equiv \left(\begin{array}{c} \mathfrak{Im}(\xi_1) \\  \mathfrak{Im}(\xi_2)  \end{array}\right)$.
The measure $\mu$ will be the Lebesgue measure on $\R^2$  
and $L^2(\R^2)$ will be the Hilbert space of all square integrable functions on $\R^2$.
A (countable) set of elements $\{\psi_i: i\in \Lambda\}$ in $L^2(\R^2)$ is called a \emph{normalized tight frame} of $L^2(\R^2)$ if
\begin{equation}\label{frame}
   \sum_{i\in\Lambda} |\langle f, \psi_i \rangle |^2 = \|f\|^2,
\forall f\in L^2(\R^2).
\end{equation}
It is well know in the literature \cite{christensen} that the equation \eqref{frame} is equivalent to
\begin{equation}\label{frameeq}
   f = \sum_{i\in\Lambda} \langle f, \psi_i \rangle \psi_i ,
\forall f\in \H.
\end{equation}
Let $\Z^2$ be the integer lattice in $\R^2.$ For a  vector $\vec{\ell}\in\Z^2,$
the \textit{translation operator} $T_{\vec{\ell}}$ is defined as
\begin{eqnarray*}
    (T_{\vec{\ell}} f)(\vec{t}) &\equiv& f(\vec{t}-\vec{\ell}), ~ \forall f\in L^2(\R^2).\\
\end{eqnarray*}
A square matrix is called \textit{expansive} if all of its eigenvalues have absolute values greater than 1.
Let $A$ be a $2 \times 2$ expansive integral matrix with eighenvalues $\{\lambda_1,\lambda_2\}$. The norm of the linear transformation $A$ on $\R^2$ (or $\C^2$) will be $\|A\|\equiv \max\{|\beta_1|,|\beta_2|\}.$
For two
vectors $\vec{t}_1,\vec{t}_2$ in the Euclidean space $\R^2,$  we have $\vec{t}_1\circ A \vec{t}_2 = A^\tau \vec{t}_1 \circ \vec{t}_2,$ where $A^\tau$ is the transpose matrix of $A.$
We define operator $U_{A}$ as
\begin{eqnarray*}
    (U_A f)(\vec{t}) &\equiv& (\sqrt{|\det(A)|}) f(A\vec{t}),~\forall f \in L^2(\R^2).
\end{eqnarray*}
This is a unitary operator on $\H.$ In particular, for an expansive matrix $A$ with $|\det(A)|=2,$ we will use
$D_A$ to denote $ U_A$ and call it \textit{dilation operator}.

\begin{definition}\label{defpsi0}
Let $A$ be an expansive integral matrix  with $|\det(A)|=2.$
A function $\psi\in\H$ is called a normalized tight frame wavelet, or NTFW, associated with $A$, if the set
\begin{equation}
    \{ D_A ^n T_{\vec{\ell}}\psi,  n\in\Z,\vec{\ell}\in\Z^2\}
\end{equation}
constitutes a normalized tight frame of $\H.$
\end{definition}
\begin{remark}
The function $\psi$ is called single function NTFW since the frame set (\ref{defpsi0})
is generated by one function $\psi.$
An NTFW is not necessarily a unit vector in $\H$ unless it is an orthonormal wavelet. By definition an element $\psi\in \H$ is an NTFW iff
\begin{equation}\label{framew}
    \|f\|^2 = \sum_{n\in\Z , \vec{\ell} \in \Z^2} |\langle f,D_A ^n T_{\vec{\ell}}\psi \rangle |^2, \forall f\in L^2((\R^2).
\end{equation}
By \eqref{frameeq} this is equivalent to
\begin{equation}\label{frameweq}
    f(\vec{t}) = \sum_{n\in\Z , \vec{\ell} \in \Z^2} \langle f,D_A ^n T_{\vec{\ell}}\psi \rangle D_A ^n T_{\vec{\ell}}\psi (\vec{t}), \forall f\in L^2(\R^2), \  \vec{t}\in\R^2 \ a.e.
\end{equation}
\end{remark}

The literatures of wavelet theory in higher dimensions are rich. Many authors provide significant contributions to the theory.
It is hard to make a short list. However, the author must cite the following names and their papers,

Q. Gu and D. Han \cite{guhan} proved that, if an integral expansive matrix associates with a single function \textit{orthogonal wavelets} with \textit{multi-resolution analysis} (MRA), then the absolute value of the matrix determinant must be $2.$ These orthogonal wavelets are special single function normalized tight frame wavelets.
In this paper we will construct single function normalized tight frame wavelets  with compact support associated with expansive integral matrices with determinant $\pm2$ in $L^2(\R^2)$.

The existance of Haar type orthonormal wavelets (hence with compact support) in $L^2(\R^2)$ was proved by J. Lagarias and
Y. Wang in \cite{wangyang2}.
The first examples of such functions with compact support and with properties of
high smoothness in  $L^2(\R^2)$ were provided by E. Belogay and Y. Wang
in  \cite{wangyang}.
The goal of this paper is to prove that every solution to the system of equations \eqref{lawtoneq} will generate filters for normalized tight frame wavelets. In applications, we just need to solve the systems of equations for filters which is feasible by exiting computer programs. Compare with \cite{wangyang2} our methods appear to be more constructive. Also, it provides variety for single function Parseval wavelets which includes the orthogonal wavelets. Also, the methods in this paper provide a wide base in searching for more frame wavelets with normal properties as wavelets in \cite{wangyang}.

We will follow the classical method for constructing such frame wavelets as provided by I. Daubechies in \cite{dau}. That is, from the filter function $m_0$ to the scaling function $\varphi$ and then to the
wavelet function $\psi$. To construct the filter function $m_0$ we start with the system of equations (\ref{lawtoneq}). The system of equations (\ref{lawtoneq}) is a generalization of
W. Lawton's system of equations \cite{lawton} for frame wavelets in $L^2(\R)$.

The scaling function $\varphi$ in this paper is not necessarily orthogonal. So the wavelet system constructed fits the definition of the
\textit{frame multi-resolution analysis} (FMRA) by   J. Benedetto and S. Li in \cite{benedetto} and  it also fits the definition of the \textit{general multi-resolution analysis} (GMRA) by
L. Baggett, H. Medina and K.  Merrill \cite{baggett}.

%

\section{Reduction Theorems}
\bigskip

In \cite{wangyang} they also prove that
every expansive $2 \times 2$ integral matrix with $|\det(A)|=2$ can be expressed in the form
 $SBS^{-1}$ where $S$ is a $2\times 2$ integral matrix with $|\det (S)|=1$ and $B$ is one of the six matrices listed below,

\begin{equation}\label{six}
\left[\begin{array}{cc} 0 & 2 \\  1 & 0 \end{array}\right],
 \left[\begin{array}{cc} 0 & 2 \\ -1 & 0 \end{array}\right],
  \left[\begin{array}{cc} 1 & 1 \\ -1 & 1 \end{array}\right],
   \left[\begin{array}{cc} -1 & -1 \\  1 & -1 \end{array}\right],
    \left[\begin{array}{cc}  0 &  2 \\ -1 &  1 \end{array}\right],
     \left[\begin{array}{cc}  0 & -2 \\  1 & -1 \end{array}\right].
\end{equation}

\begin{proposition}\label{linktonew6}
Let $A$ be a $2\times 2$ expansive integral matrix with $|\det (A)| = 2.$
Then there is a $2\times 2$ integral matrix $S$ with $|\det (S)| =1$ such that
$S A S^{-1}$ is one of the following six matrices,
\begin{equation}\label{newsix}
\left[\begin{array}{cc} 1 & 1 \\  1 & -1 \end{array}\right],
 \left[\begin{array}{cc} 1 & -3 \\ 1 & -1 \end{array}\right],
  \left[\begin{array}{cc} 1 & 1 \\ -1 & 1 \end{array}\right],
   \left[\begin{array}{cc} -1 & -1 \\  1 & -1 \end{array}\right],
    \left[\begin{array}{cc}  -1 &  2 \\ -2 &  2 \end{array}\right],
     \left[\begin{array}{cc}  1 & -2 \\  2 & -2 \end{array}\right]
\end{equation}
\end{proposition}
\begin{proof}
This statement is an immediate consequence of the list (\ref{six}) by
E. Belogay and Y. Wang \cite{wangyang} and
the following calculation.
\begin{eqnarray*}
\left[\begin{array}{cc} -1 & 1 \\  2 & -3 \end{array}\right]
\left[\begin{array}{cc} 0 & 2 \\  1 & 0 \end{array}\right]
\left[\begin{array}{cc} -1 & 1 \\  2 & -3 \end{array}\right]^{-1}&=&
\left[\begin{array}{cc} 1 & 1 \\  1 & -1 \end{array}\right],\\
 \left[\begin{array}{cc} 1 & -1 \\ 0 & -1 \end{array}\right]
\left[\begin{array}{cc} 0 & 2 \\  -1 & 0 \end{array}\right]
 \left[\begin{array}{cc} 1 & -1 \\ 0 & -1 \end{array}\right]^{-1}&=&
 \left[\begin{array}{cc} 1 & -3 \\ 1 & -1 \end{array}\right],\\
  \left[\begin{array}{cc}  -1 &  1 \\ -1 &  0 \end{array}\right]
\left[\begin{array}{cc}  0 &  2 \\ -1 &  1 \end{array}\right]
  \left[\begin{array}{cc}  -1 &  1 \\ -1 &  0 \end{array}\right]^{-1}&=&
    \left[\begin{array}{cc}  -1 &  2 \\ -2 &  2 \end{array}\right],\\
 \left[\begin{array}{cc}  -1 & 1 \\  -1 & 0 \end{array}\right]
\left[\begin{array}{cc}  0 & -2 \\  1 & -1 \end{array}\right]
 \left[\begin{array}{cc}  -1 & 1 \\  -1 & 0 \end{array}\right]^{-1}&=&
\left[\begin{array}{cc}  1 & -2 \\  2 & -2 \end{array}\right].
\end{eqnarray*}

\end{proof}

\begin{lemma}\label{1}
Let $A$ be a $2\times 2$ expansive integral matrix with $|\det(A)| =2$.
For a $2\times 2$ integral matrix $S$ of $|\det(S)|=1,$ assume $B=S^{-1}AS$. Then
\begin{eqnarray}
\label{P} U_S T_{\vec{\ell}}U_S^{-1} &=& T_{S^{-1}\vec{\ell}}, \forall \vec{\ell}\in\Z^2;\\
\label{Q} U_S D_A ^n U_S^{-1} &=& D_B ^n, \forall n\in\Z.
\end{eqnarray}
\end{lemma}

\begin{proof}
Let $h\in L^2 (\R^2).$ By definition, $U_S U_{S^{-1}} h (\vec{t})=U_S h(S^{-1}\vec{t})=h(SS^{-1}\vec{t}) =h(\vec{t})$, so $U_S U_{S^{-1}}=I.$
Similarly, we have $U_{S^{-1}}U_S =I.$ Therefore, $U_S ^{-1}=U_{S^{-1}}.$ For $\vec{\ell}\in\Z^d,$ we have
\begin{eqnarray*}
 U_S T_{\vec{\ell}} U_S ^{-1} h(\vec{t})
&=&  U_S T_{\vec{\ell}} U_{S ^{-1}} h(\vec{t}) \\
    &=& U_S T_{\vec{\ell}} h( S^{-1} \vec{t})\\
        &=& U_S  h( S^{-1} (\vec{t}-\vec{\ell}))\\
            &=& U_S  h( S^{-1} \vec{t}-S^{-1}\vec{\ell})\\
                &=&   h( S^{-1}S \vec{t}-S^{-1}\vec{\ell})\\
                    &=&   h( \vec{t}-S^{-1}\vec{\ell})\\
                        &=&   T_{S^{-1}\vec{\ell}} h( \vec{t}).
\end{eqnarray*}
So we have equation (\ref{P}).
Also we have
$U_S D_A U_S^{-1} h (\vec{t}) = \sqrt{2} h(SAS^{-1}\vec{t})=D_B h(\vec{t}).$
So,
\begin{eqnarray*}
  U_S D_A U_S^{-1} &=& D_B \\
  U_S D_A ^{-1} U_S^{-1} &=& (U_S D_A U_S^{-1})^{-1}=D_B^{-1}.
\end{eqnarray*}
This implies that $\forall n\in\N$,
\begin{eqnarray*}
U_S D_A ^n U_S^{-1}&=&(U_S D_A U_S^{-1})^n=D_B ^n;\\
U_S D_A ^{-n} U_S^{-1}&=&(U_S D_A U_S^{-1})^{-n}=D_B ^{-n}.
\end{eqnarray*}
This proves equation (\ref{Q}).

\end{proof}

\begin{theorem}\label{redsym}
Let $A$ be a $2\times 2$ expansive integral matrix with $|\det(A)| =2$ and $S$ be a $2\times 2$ integral matrix with property  that $|\det(S)|=1$. Let $B\equiv S^{-1}AS$.  Assume that
a function $\psi_A$ is a normalized tight frame wavelet associated with the matrix $A.$ Then the function $\eta_B \equiv U_S \psi_A$ is a normalized tight frame wavelet associated with the matrix $B.$
\end{theorem}

\begin{proof}
By assumption and Lemma \ref{1} we have   $B=S^{-1}AS, D_B=U_B=U_{S^{-1}AS} = U_{S^{-1}} U_A U_{S}= U_S ^{-1}D_A U_S$.
Let $f\in L^2 (\R^2).$ We have

\begin{eqnarray*}
  U_S ^{-1}  f &=& \sum _{n\in \Z, \vec{\ell}\in\Z^2}
  \langle   U_S ^{-1} f, D_A ^n  T_{\vec{\ell}} \psi_A \rangle D_A ^n  T_{\vec{\ell}} \psi_A.
  \end{eqnarray*}
Since $U_S$ is a unitary operator, we have

\begin{eqnarray*}
 f  &=& \sum _{n\in \Z, \vec{\ell}\in\Z^2}
 \langle   f, U_S D_A ^n  T_{\vec{\ell}} \psi_A \rangle
 U_S D_A ^n  T_{\vec{\ell}}\psi_A\\
 &=& \sum _{n\in \Z, \vec{\ell}\in\Z^2}
 \langle   f, U_S D_A ^n U_S ^{-1} U_S  T_{\vec{\ell}} U_S ^{-1} U_S \psi_A \rangle
 U_S D_A ^n U_S ^{-1} U_S  T_{\vec{\ell}} U_S ^{-1} U_S \psi_A\\
&=& \sum _{n\in \Z, \vec{\ell}\in\Z^2}
 \langle   f, D_B ^n  T_{S^{-1}\vec{\ell}} \eta_B \rangle
 D_B ^n  T_{S^{-1}\vec{\ell}}\eta_B\\
 &=& \sum _{n\in \Z, \vec{\ell}\in S^{-1}\Z^2}
 \langle   f, D_B ^n  T_{\vec{\ell}} \eta_B \rangle
 D_B ^n  T_{\vec{\ell}}\eta_B
 \end{eqnarray*}
Since $S$ is an integral matrix with  $|\det (S)|=1,$ we have $\Z^2 = S \Z^2 = S^{-1} \Z^2.$
So we have
\begin{eqnarray*}
f &=& \sum _{n\in \Z, \vec{\ell}\in\Z^2}
 \langle   f, D_B ^n  T_{\vec{\ell}} \eta_B \rangle
 D_B ^n  T_{\vec{\ell}}\eta_B
\end{eqnarray*}

\end{proof}

For $f,g\in L^1(\R^2) \cap L^2(\R^2),$ the Fourier Transform and Fourier Inverse Transform are defined as
\begin{eqnarray*}
(\F f)(\vec{s}) &=& \frac{1}{2\pi} \int_{\R^2}e^{-i\vec{s}\circ\vec{t}}f(\vec{t})d\vt=\hat{f}(\vec{s}),\\
(\F^{-1} g)(\vec{t})&=& \frac{1}{2\pi}\int_{\R^2}e^{i\vec{s}\circ\vec{t}}g(\vec{s})d\vec{s}=\check{g}(\vec{t}).\\
\end{eqnarray*}
The set $L^1(\R^2) \cap L^2(\R^2)$ is dense in $\H,$ the operator $\F$ extends to a unitary operator on $\H$ which is still called Fourier Transform.
For an operator $V$ on $\H,$ we will write $\F V \F^{-1} \equiv \widehat{V}.$
We will use the following formulas in this paper.
\begin{lemma}\label{k}
Let $A$ be a $2\times 2$ expansive integral matrix, then
\begin{eqnarray*}
  T_{\vec{\ell}} D_A &=& D_AT_{A\vec{\ell}},\\
  \widehat{T}_{\vec{\ell}} &=& M_{e^{- i \vec{s}\circ\vec{\ell}}},\\
  \widehat{D}_A &=& U_{{(A^{-1})}^\tau} = U_{{(A^\tau)}^{-1}}=D_{A^\tau}^{-1}=D_{A^\tau}^*,
\end{eqnarray*}
where $M_{e^{- i \vec{s}\circ\vec{\ell}}}$ is the multiplication operator by $e^{- i \vec{s}\circ\vec{\ell}}$.
Operators $T_{\vec{\ell}},D_A ,\F$ and $M_{e^{- i \vec{s}\circ\vec{\ell}}}$ are unitary operators acting on $\H$.\\
\end{lemma}
\begin{remark}\label{vv}
For the translation operator $T_{A^{-J}\vec{\ell}}$ where   vector $A^{-J}\vec{\ell}$
 is in the refined lattice $A^{-J}\Z^2$ we have
 $T_{A^{-J}\vec{\ell}} D_A^{J}=D_A ^J T_{\vec{\ell}}.$ We also have
$\overline{\widehat{D}_A^{J}\widehat{\varphi}(\vt)} =
\frac{1}{\sqrt{2^J}}\overline{\widehat{\varphi}((A^\tau)^{-J}\vt)}.$ We will need this in the proof of Lemma \ref{identitylj}.
We leave these to the reader to verify, using the same method as in the proof of Lemma \ref{k}.
\end{remark}
\begin{proof}
We have
\begin{eqnarray*}
(\widehat{T}_{\vec{\ell}}\hat{f})(\vec{s})
&=& (\F T_{\vec{\ell}} \F^{-1} \F f)(\vec{s}) \\ &=&  \frac{1}{2\pi}\int_{\R^2}e^{-i\vec{s}\circ \vec{t}}f(\vec{t}-\vec{\ell})d\vt,\\
&=& e^{i\vec{s}\circ\vec{\ell}} \cdot \frac{1}{2\pi} \int_{\R^2}e^{-i\vec{s}\circ\vec{u}}f(\vec{u})d\vec{u},\\
&=& e^{i\vec{s}\circ\vec{\ell}} \cdot \hat{f}(\vec{s})\\
&=& (M_{e^{i\vec{s}\circ\vec{\ell}}}\hat{f})(\vec{s}).\\
\end{eqnarray*}
Here the substitution $\vec{u}=\vec{t}-\vec{\ell}$ is used.
Next, we have
\begin{eqnarray*}
(\widehat{D}_{A}\hat{f})(\vec{s})
&=& (\F D_{A}f)(\vec{s}) \\
&=& \frac{1}{2\pi} \int_{\R^2}e^{-i\vec{s}\circ\vec{t}}\cdot \sqrt{2} f(A\vec{t})d\mu\\
&=& \frac{1}{\sqrt{2}} \cdot \frac{1}{2\pi}  \int_{\R^2}e^{-i\vec{s}\circ(A^{-1}\vec{u})}f(\vec{u})d\nu\\
&=& \frac{1}{\sqrt{2}} \cdot \frac{1}{2\pi}  \int_{\R^2}e^{-i{({A^{-1}})^\tau}\vec{s}\circ\vec{u}}f(\vec{u})d\nu.
\end{eqnarray*}
Here substitutions $\vec{u}=A\vec{t}$  and $d\nu = 2 d \mu$ are used.
 So we have
\begin{eqnarray*}
(\widehat{D}_{A}\hat{f})(\vec{s})
&=& \sqrt{\det((A^{-1})^\tau)} \cdot \hat{f}(({A^{-1}})^\tau\vec{s})\\
&=& (U_{{(A^{-1})}^\tau} \hat{f})(\vec{s}).
\end{eqnarray*}
This implies that
\begin{eqnarray*}
\widehat{D}_{A}
&=& U_{{(A^{-1})}^\tau} = U_{{(A^\tau)}^{-1}}=D_{A^\tau}^{-1}=D_{A^\tau}^*.
\end{eqnarray*}
Also, we have
\begin{eqnarray*}
T_{\vec{\ell}} D_A f(\vec{t})
&=& \sqrt{2} T_{\vec{\ell}} f(A\vec{t}) \\
&=& \sqrt{2} f(A(\vec{t}-\vec{\ell})) \\
&=& \sqrt{2} f(A\vec{t}-A\vec{\ell})) \\
&=& D_AT_{A\vec{\ell}} f(\vec{t}).\\
T_{\vec{\ell}} D_A &=& D_AT_{A\vec{\ell}}.
\end{eqnarray*}
\end{proof}

The integral lattice $\Z^2$ is an Abelian group under vector addition. The subset $(2\Z)^2$ is a subgroup. For a fixed $2\times 2$ integral matrix $A$ with $|\det(A)| =2,$
the two sets $A \Z^2$ and $A^\tau \Z^2$ are proper subgroups of $\Z^2$ containing $(2\Z)^2$. The two quotient groups  $\frac{A\Z^2}{(2\Z)^2}$ and $\frac{A^\tau \Z^2}{(2\Z)^2}$ are two proper subgroups of the quotient group $\frac{\Z^2}{(2\Z)^2}$ which has $4$ elements,
$\Big\{ \left(\begin{array}{c} 0 \\  0  \end{array}\right)+(2\Z)^2,
\left(\begin{array}{c} 0 \\  1  \end{array}\right)+(2\Z)^2,\left(\begin{array}{c} 1 \\  0  \end{array}\right)+(2\Z)^2,\left(\begin{array}{c} 1 \\  1  \end{array}\right)+(2\Z)^2 \Big\}.$
If the two elements of the subgroup $\frac{A \Z^2}{(2\Z)^2}$ are  $\vec{0}+(2\Z)^2, \vec{s}+(2\Z)^2,$ we will call $\{\vec{0},\vec{s}\}$ the \textit{generators} for $A\Z^2.$
 We define the generators for $A^\tau \Z^2$ in the similar way. $A\Z^2 = A^\tau\Z^2$ if and only if they have the same generators (in the $4$ elements).

\begin{proposition}\label{properties}
Let $A$ be one of the six matrices in (\ref{newsix}) as in Proposition \ref{linktonew6}.
Then there exist vectors $\vec{\ell}_A$ and $\vec{q}_A$ in $\Z^2$ with the following properties,
\begin{enumerate}
  \item $\Z^2 = A^\tau \Z^2 \dotcup (\vec{\ell}_A + A^\tau \Z^2)$;
  \item $\vec{q}_A \circ A^\tau \Z^2 \subseteq 2\Z$ and
        $\vec{q}_A \circ (\vec{\ell}_A + A^\tau \Z^2) \subseteq 2\Z+1$;
  \item\label{2z} $A^\tau \vec{q}_A \in (2\Z)^2.$
  \item\label{aatau} $A \Z^2 = A^\tau \Z^2.$
\end{enumerate}
\end{proposition}
\begin{remark}
Equation  (\ref{aatau}) is not true in general. Let $A$ be $\left[\begin{array}{cc}  0 & -2 \\  1 & -1 \end{array}\right]$ which is in the list (\ref{six}). Then $A^\tau = \left[\begin{array}{cc}  0 & 1 \\  -2 & -1 \end{array}\right].$
It is left to the reader to verify that $\Big\{ \left(\begin{array}{c} 0 \\  0  \end{array}\right),
\left(\begin{array}{c} 0 \\  1  \end{array}\right) \Big\}$ is the generator for $A^\tau\Z^2$ while
 the generator for $A\Z^2$ is
$\Big\{ \left(\begin{array}{c} 0 \\  0  \end{array}\right),
\left(\begin{array}{c} 1 \\  1  \end{array}\right) \Big\}$.
So $A\Z^2 \neq A^\tau \Z^2.$
\end{remark}
\begin{proof}
1.  Let
$A=\left[\begin{array}{cc} 1 & 1 \\  1 & -1 \end{array}\right].$ Then
$A^\tau=\left[\begin{array}{cc} 1 & 1 \\  1 & -1 \end{array}\right].$
It is left to the reader to verify that
$\Big\{ \left(\begin{array}{c} 0 \\  0  \end{array}\right),
\left(\begin{array}{c} 1 \\  1  \end{array}\right) \Big\}$ is the generator for both
$A\Z^2$ and $A^\tau \Z^2.$
Therefore, we have equation (4) $A\Z^2=A^\tau \Z^2$
\footnote{In this example, $A=A^\tau.$ We do not assume this condition in general .}
since
\begin{eqnarray*}
  A \Z^2  &=& \Big(\left(\begin{array}{c} 0 \\  0  \end{array}\right)+(2\Z)^2\Big) \dotcup \Big(\left(\begin{array}{c} 1 \\  1  \end{array}\right)+(2\Z)^2\Big), \\
  A^\tau \Z^2  &=& \Big(\left(\begin{array}{c} 0 \\  0  \end{array}\right)+(2\Z)^2\Big) \dotcup \Big(\left(\begin{array}{c} 1 \\  1  \end{array}\right)+(2\Z)^2\Big).
\end{eqnarray*}
This also implies that
\begin{equation*}
\Z^2\backslash A^\tau \Z^2 = \Big(\left(\begin{array}{c} 1 \\  0  \end{array}\right)+(2\Z)^2\Big) \dotcup \Big(\left(\begin{array}{c} 0 \\  1  \end{array}\right)+(2\Z)^2\Big)=\left(\begin{array}{c} 1 \\  0  \end{array}\right) +A^\tau\Z^2,
\end{equation*}
since $\big\{
   \Big(\left(\begin{array}{c} 0 \\  0  \end{array}\right)+(2\Z)^2\Big), \Big(\left(\begin{array}{c} 1 \\  1  \end{array}\right)+(2\Z)^2\Big),
       \Big(\left(\begin{array}{c} 1 \\  0  \end{array}\right)+(2\Z)^2\Big), \Big(\left(\begin{array}{c} 0 \\  1  \end{array}\right)+(2\Z)^2\Big)
 \big\}$ is a partition of $\Z^2.$
So, the vector
 $\vec{\ell}_A\equiv\left(\begin{array}{c} 1 \\  0  \end{array}\right)$
satisfies the equation
\begin{equation*}
\Z^2 = A^\tau \Z^2 \dotcup (\vec{\ell}_A + A^\tau \Z^2).
\end{equation*}
Define $\vec{q}_A \equiv \left(\begin{array}{c} 1 \\  1  \end{array}\right).$ It is left to the reader to verify that $q_A \circ A^\tau \Z^2$ are even numbers and
$\vec{q}_A \circ (\vec{\ell}_A +A^\tau \Z^2)$ are odd numbers since $\vec{q}_A \circ \vec{\ell}_A=1.$ Finally, $A^\tau \vec{q}_A = \left[\begin{array}{cc} 1 & 1 \\  1 & -1 \end{array}\right] \left(\begin{array}{c} 1 \\  1  \end{array}\right)=\left(\begin{array}{c} 2 \\  0 \end{array}\right) \in (2\Z)^2.$ This proves property (3).

2. We list all six matrices in the list (\ref{newsix}) and their corresponding $\vec{\ell}_A$ and $\vec{q}_A$ in the next table. The reader may verify as we did in part 1, the equations in (1),(2) and (3) are satisfied. Also, it is left to the reader to verify that for each matrix in the six cases, the generators for $A\Z^2$ and $A^\tau \Z^2$ are the same. So we have now established the equation in (\ref{aatau}).

\begin{equation*}
\begin{array}{cccccc}
  A & A^\tau & \text{gen of } A^\tau \Z^2 & \vec{\ell}_A & \vec{q}_A & A^\tau \vec{q}_A\\
 &  &  &  &  & \\
  \left[\begin{array}{cc} 1 & 1 \\  1 & -1 \end{array}\right]   &
  \left[\begin{array}{cc} 1 & 1 \\  1 & -1 \end{array}\right]   &
  \left(\begin{array}{c} 0 \\  0  \end{array}\right),
\left(\begin{array}{c} 1 \\  1  \end{array}\right)              &
\left(\begin{array}{c} 1 \\  0  \end{array}\right)              &
\left(\begin{array}{c} 1 \\  1  \end{array}\right)              &
\left(\begin{array}{c} 2 \\  0 \end{array}\right)               \\
 &  &  &  &  & \\
  \left[\begin{array}{cc} 1 & -3 \\  1 & -1 \end{array}\right]   &
  \left[\begin{array}{cc} 1 & 1 \\  -3 & -1 \end{array}\right]   &
  \left(\begin{array}{c} 0 \\  0  \end{array}\right),
\left(\begin{array}{c} 1 \\  1  \end{array}\right)              &
\left(\begin{array}{c} 1 \\  0  \end{array}\right)              &
\left(\begin{array}{c} 1 \\  1  \end{array}\right)              &
\left(\begin{array}{c} 2 \\  -4 \end{array}\right)               \\
 &  &  &  &  & \\
  \left[\begin{array}{cc} 1 & 1 \\  -1 & 1 \end{array}\right]   &
  \left[\begin{array}{cc} 1 & -1 \\  1 & 1 \end{array}\right]   &
  \left(\begin{array}{c} 0 \\  0  \end{array}\right),
\left(\begin{array}{c} 1 \\  1  \end{array}\right)              &
\left(\begin{array}{c} 1 \\  0  \end{array}\right)              &
\left(\begin{array}{c} 1 \\  1  \end{array}\right)              &
\left(\begin{array}{c} 0 \\  2 \end{array}\right)               \\
 &  &  &  &  & \\
  \left[\begin{array}{cc} -1 & -1 \\  1 & -1 \end{array}\right]   &
  \left[\begin{array}{cc} -1 & 1 \\  -1 & -1 \end{array}\right]   &
  \left(\begin{array}{c} 0 \\  0  \end{array}\right),
\left(\begin{array}{c} 1 \\  1  \end{array}\right)              &
\left(\begin{array}{c} 1 \\  0  \end{array}\right)              &
\left(\begin{array}{c} 1 \\  1  \end{array}\right)              &
\left(\begin{array}{c} 0 \\  -2 \end{array}\right)               \\
 &  &  &  &  & \\
  \left[\begin{array}{cc} -1 & 2 \\  -2 & 2 \end{array}\right]   &
  \left[\begin{array}{cc} -1 & -2 \\  2 & 2 \end{array}\right]   &
  \left(\begin{array}{c} 0 \\  0  \end{array}\right),
\left(\begin{array}{c} 1 \\  0  \end{array}\right)              &
\left(\begin{array}{c} 0 \\  1  \end{array}\right)              &
\left(\begin{array}{c} 0 \\  1  \end{array}\right)              &
\left(\begin{array}{c} -2 \\ 2 \end{array}\right)               \\
 &  &  &  &  & \\
  \left[\begin{array}{cc} 1 & -2 \\  2 & -2 \end{array}\right]   &
  \left[\begin{array}{cc} 1 & 2 \\  -2 & -2 \end{array}\right]   &
  \left(\begin{array}{c} 0 \\  0  \end{array}\right),
\left(\begin{array}{c} 1 \\  0  \end{array}\right)              &
\left(\begin{array}{c} 0 \\  1  \end{array}\right)              &
\left(\begin{array}{c} 0 \\  1  \end{array}\right)              &
\left(\begin{array}{c} 2 \\  -2 \end{array}\right)               \\
\end{array}
\end{equation*}
\end{proof}

Two subsets $\mathcal{G}_1, \mathcal{G}_2$ of $\R^2$ are said to be $2$-{\it translation-equivalent}, or
$\mathcal{G}_1 \stackrel{2}{\sim} \mathcal{G}_2$, if
there exists a  mapping  $\Theta$ from $\mathcal{G}_1$ onto $\mathcal{G}_2$ with the property that
\begin{equation*}
    \Theta (\vec{t}) -\vec{t} \in (2\Z)^2, \ \vec{t}\in \mathcal{G}_1 \text{ a.e.}
\end{equation*}

\begin{proposition}\label{Ataugamma}
Let $A$ be one of the six expansive matrices in Proposition \ref{properties}, $\vec{q}_A$ be the corresponding vector related to $A$
and $\Gamma_0\equiv [-1,1]^2.$ Then, there are two measurable sets $\Gamma_1$ and $\Gamma_2$ such that
\begin{eqnarray*}
  \Gamma_1 &\stackrel{2}{\sim}& \Gamma_0; \\
  \Gamma_2 &\stackrel{2}{\sim}& \Gamma_0; \\
  A^\tau \Gamma_0 &\stackrel{2}{\sim}& \Gamma_1 \dotcup (\vec{q}_A + \Gamma_2). \\
\end{eqnarray*}
\end{proposition}
\begin{corollary}\label{click}
Let $\Gamma_\pi \equiv \pi \Gamma_0 = [-\pi,\pi]^2,$ and $h(\vx)$ be a $2\pi$-periodical continuous function on $\R^2.$ Then
\begin{equation}
\int _{A^\tau \Gamma_\pi} h(\vx)d \mu = \int _{\Gamma_\pi} h(\vx)d \mu + \int _{\Gamma_\pi + \pi \vec{q}_A} h(\vx)d \mu
\end{equation}
\end{corollary}

\begin{proof}
1. For any matrix $A$ in the collection
$\left[\begin{array}{cc} 1 & 1 \\  1 & -1 \end{array}\right] $,
$\left[\begin{array}{cc} 1 & 1 \\  -1 & 1 \end{array}\right]$ and
$\left[\begin{array}{cc} -1 & -1 \\  1 & -1 \end{array}\right]$,  $A^\tau \Gamma_0$ has the same
vertices of
$\left\{\left(\begin{array}{c} 2 \\  0  \end{array}\right),
\left(\begin{array}{c} -2 \\  0  \end{array}\right),
\left(\begin{array}{c} 0 \\  2  \end{array}\right),
\left(\begin{array}{c} 0 \\  -2  \end{array}\right)\right\}$.
By the table in the proof of Proposition \ref{properties}, the above three matrices share the same vector
$\vec{q}_A=\left(\begin{array}{c} 1 \\  1  \end{array}\right)$ and $\Gamma_0 \subset A^\tau \Gamma_0$ (Figure 1 left).
It is enough to discuss only one of the three cases.

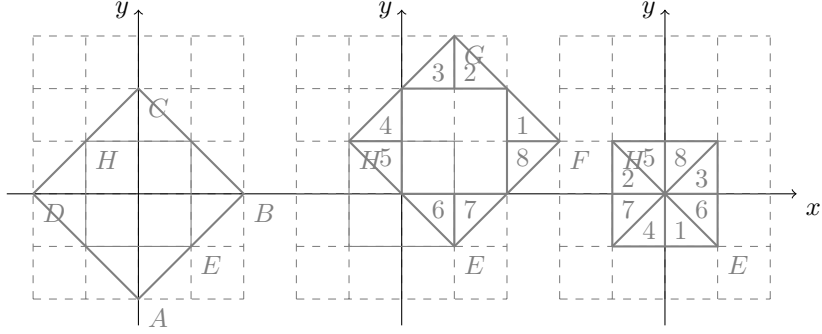
\begin{figure}[h]
\begin{tikzpicture}[scale=.7]
\draw[step=1,gray,dashed] (-2,-2) grid (2,3);
\draw[->] (-2.5,0)-- (12.5,0)
node[below right] {$x$};
\draw[->] (0,-2.5)-- (0,3.5)
node[left] {$y$};
\draw[step=.25cm,gray,thick] (-2,0)-- (0,-2)node[below right] {$A$}-- (2,0)node[below right] {$B$}-- (0,2)node[below right] {$C$}-- (-2,0)node[below right] {$D$};
\draw[step=.25cm,gray] (-1,-1)-- (1,-1)node[below right] {$E$}-- (1,1)-- (-1,1)node[below right] {$H$}-- (-1,-1);

\draw[step=1,gray,dashed] (3,-2) grid (7,3);
\draw[->] (5,-2.5)-- (5,3.5)
node[left] {$y$};
\draw[step=.25cm,gray,thick] (-2+5+1,0+1)-- (0+5+1,-2+1)node[below right] {$E$}-- (2+5+1,0+1)node[below right] {$F$}-- (0+5+1,2+1)node[below right] {$G$}-- (-2+5+1,0+1)node[below right] {$H$};
\draw[step=.25cm,gray] (-1+5,-1)-- (1+5,-1)-- (1+5,1)-- (-1+5,1)-- (-1+5,-1);
\draw[step=.25cm,gray,thick] (5,0)-- (5,2);
\draw[step=.25cm,gray,thick] (4,1)-- (5,1);
\draw[step=.25cm,gray,thick] (5,2)-- (7,2);
\draw[step=.25cm,gray,thick] (6,2)-- (6,3);
\draw[step=.25cm,gray,thick] (7,0)-- (7,2);
\draw[step=.25cm,gray,thick] (7,1)-- (8,1);
\draw[step=.25cm,gray,thick] (5,0)-- (7,0);
\draw[step=.25cm,gray,thick] (6,0)-- (6,-1);
\draw[step=1cm,gray] (7.3,1.3)
node {1};
\draw[step=1cm,gray] (6.3,2.3)
node {2};
\draw[step=1cm,gray] (5.7,2.3)
node {3};
\draw[step=1cm,gray] (4.7,1.3)
node {4};
\draw[step=1cm,gray] (4.7,0.7)
node {5};
\draw[step=1cm,gray] (5.7,-0.3)
node {6};
\draw[step=1cm,gray] (6.3,-0.3)
node {7};
\draw[step=1cm,gray] (7.3,0.7)
node {8};

\draw[step=1,gray,dashed] (8,-2) grid (12,3);
\draw[->] (10,-2.5)-- (10,3.5)
node[left] {$y$};
\draw[step=.25cm,gray,thick] (-1+10,-1)-- (1+10,-1)node[below right] {$E$}-- (1+10,1)-- (-1+10,1)-- (-1+10,-1);
\draw[step=.25cm,gray,thick] (-1+10,-1)-- (1+10,1);
\draw[step=.25cm,gray,thick] (1+10,-1)-- (-1+10,1)node[below right] {$H$};
\draw[step=.25cm,gray,thick] (-1+10,0)-- (1+10,0);
\draw[step=.25cm,gray,thick] (10,-1)-- (10,1);
\draw[step=1cm,gray] (10.7,0.3)
node {3};
\draw[step=1cm,gray] (9.7,-0.7)
node {4};
\draw[step=1cm,gray] (10.3,-0.7)
node {1};
\draw[step=1cm,gray] (10.7,-0.3)
node {6};
\draw[step=1cm,gray] (10.3,0.7)
node {8};
\draw[step=1cm,gray] (9.7,0.7)
node {5};
\draw[step=1cm,gray] (9.3,0.3)
node {2};
\draw[step=1cm,gray] (9.3,-0.3)
node {7};
\end{tikzpicture}
\caption{$A ^\tau \Gamma_0$}
\end{figure}
 Consider $A= \left[\begin{array}{cc} 1 & 1 \\  1 & -1 \end{array}\right].$
Let $\Gamma_1 \equiv \Gamma_0$ and $\Gamma_2 \equiv A^\tau \Gamma_0 \backslash \Gamma_1.$
Notice that $\Gamma_2+\vec{q}_A$ (Figure 1, middle) is a disjoint union of eight triangles $\{\triangle_k,k=1,2,\cdots,8\}$. The following new triangles $\{\triangle_k ^\prime,k=1,2,\cdots,8\}$ form a partition for $\Gamma_0$ modulus zero measure sets (Figure 1, right).
We have
$
\triangle_1^\prime = \triangle_1 + \left(\begin{array}{c} -2 \\  -2  \end{array}\right) ;
\triangle_2^\prime = \triangle_2 + \left(\begin{array}{c} -2 \\  -2  \end{array}\right) ;
\triangle_3^\prime = \triangle_3 + \left(\begin{array}{c} 0 \\  -2  \end{array}\right) ;
\triangle_4^\prime = \triangle_4 + \left(\begin{array}{c} 0 \\  -2  \end{array}\right) ;
\triangle_5^\prime = \triangle_5;
\triangle_6^\prime = \triangle_6;
\triangle_7^\prime = \triangle_7 + \left(\begin{array}{c} -2 \\  0  \end{array}\right) ;
\triangle_8^\prime = \triangle_8 + \left(\begin{array}{c} -2 \\  0  \end{array}\right) .
$
This proves that
$\vec{q}_A + \Gamma_2\stackrel{2}{\sim} \Gamma_0.$

\begin{figure}[h]
\begin{tikzpicture}[scale=.7]
\draw[->] (-2.5,0)-- (12.5,0)
node[below right] {$x$};
\draw[->] (0,-4.5)-- (0,4.5)
node[left] {$y$};
\draw[step=1,gray,dashed] (-2,-4) grid (2,4);
\draw[step=.25cm,gray,thick] (2,-4)node[below right] {$A$}-- (0,2)node[below right] {$B$}-- (-2,4)node[below right] {$C$}-- (0,-2)node[below right] {$D$}-- (2,-4);
\draw[step=.25cm,gray] (-1,-1)-- (1,-1)-- (1,1)-- (-1,1)-- (-1,-1);
\draw[step=.25cm,gray,thick] (0,-2)-- (0,2);

\draw[step=1,gray,dashed] (3,-4) grid (7,4);
\draw[->] (5,-4.5)-- (5,4.5)
node[left] {$y$};
\draw[step=.25cm,gray] (-1+5,-1)-- (1+5,-1)-- (1+5,1)-- (-1+5,1)-- (-1+5,-1);
\draw[step=.25cm,gray,thick] (5,2)node[below right] {$B$}-- (3,4)node[below right] {$C$}
-- (3,0)node[below right] {$E$}-- (5,-2)node[below right] {$D$}-- (5,2);
\draw[step=1,gray,dashed,thick] (3,2)node[below right] {$F$}-- (5,2);
\draw (3,-2)node[below right] {$G$};
\draw[step=.25cm,gray,dashed,thick] (-2,4)-- (-2,0)node[below right] {$E$} --(0,-2);

\draw[->] (10,-4.5)-- (10,4.5)
node[left] {$y$};
\draw[step=1,gray,dashed] (8,-4) grid (12,4);
\draw[step=.25cm,gray] (-1+10,-1)-- (1+10,-1)-- (1+10,1)-- (-1+10,1)-- (-1+10,-1);
\draw[step=.25cm,gray,thick] (-2+10,-2)node[below right] {$G$}-- (-2+10,2)node[below right] {$F$}-- (0+10,2)node[below right] {$B$}-- (0+10,-2)node[below right] {$D$}-- (-2+10,-2);
\draw[step=.25cm,gray,thick] (-2+10,1)node[below right] {$H$};
\draw[step=.25cm,gray,thick] (-2+10+1,1)node[below right] {$I$};
\draw[step=.25cm,gray,thick] (-2+10+2,1)node[below right] {$J$};
\draw[step=.25cm,gray,thick] (-2+10,-1)node[below right] {$K$};
\draw[step=.25cm,gray,thick] (-2+10+1,-1)node[below right] {$L$};
\draw[step=.25cm,gray,thick] (-2+10+2,-1)node[below right] {$M$};
\draw[step=.25cm,gray,thick] (-2+10,1)-- (0+10,1);
\draw[step=.25cm,gray,thick] (-2+10,-1)-- (0+10,-1);
\draw[step=.25cm,gray,thick] (-2+10,-1)-- (-2+10,1);
\draw[step=.25cm,gray,thick,dashed] (5,-2)-- (3,-2)-- (3,0);
\end{tikzpicture}
\caption{$A^\tau \Gamma_0$}
\end{figure}

2. Let $A=\left[\begin{array}{cc} 1 & -3 \\  1 & -1 \end{array}\right].$
Then $A^\tau = \left[\begin{array}{cc} 1 & 1 \\  -3 & -1 \end{array}\right]$ and
$\vec{q}_A=\left(\begin{array}{c} 1 \\  1  \end{array}\right)$.
$A^\tau \Gamma_0$ is the parallelogram $ABCD$ (Figure 2, left). It is $2$-translation equivalent to parallelogram $BCED$ (Figure 2, middle) which is
the disjoint union of $\triangle ABD + \left(\begin{array}{c} -2 \\  2  \end{array}\right)$ and $\triangle CBD.$
The parallelogram $BCED$ is $2$-translation equivalent to rectangle $BFGD$ since
$\triangle DGE = \triangle BFC +\left(\begin{array}{c} 0 \\  -2  \end{array}\right).$
Now let $\Gamma_1$ be the square $MJHK.$ It is $2$-translation equivalent to $\Gamma_0$ since
$\Gamma_0 = MJIL\dotcup (LIHK + \left(\begin{array}{c} 0 \\  2  \end{array}\right)).$
Let $\Gamma_2 \equiv \Box JBFH \dotcup \Box DMKG.$ Thus
$\Big(\Big(\Box JBFH + \left(\begin{array}{c} 0 \\  -2  \end{array}\right) \Big)\dotcup \Box DMKG\Big)+ \vec{q}_A =\Gamma_0.$

\begin{figure}[h]
\begin{tikzpicture}[scale=.7]
\draw[->] (-3.5,0)-- (12.5,0)
node[below right] {$x$};
\draw[->] (0,-4.5)-- (0,4.5)
node[right] {$y$};
\draw[->] (6,-4.5)-- (6,4.5)
node[right] {$y$};
\draw[->] (10,-4.5)-- (10,4.5)
node[right] {$y$};
\draw[step=1,gray,dashed] (-3,-4) grid (12,4);
\draw[step=.25cm,gray,thick] (-3,4)-- (-1,0)node[below right] {$D$}-- (3,-4)node[below right] {$A$}-- (1,0)node[below right] {$B$}-- (-3,4)node[below right] {$C$};
\draw[step=.25cm,gray,thick] (-3+6,4)-- (-1+6,4)node[below right] {$E$}-- (1+6,0)node[below right] {$B$}-- (-1+6,0)node[below right] {$D$}-- (-3+6,4)node[below right] {$C$};
\draw[step=.25cm,gray,thick] (-3+10+2,4)node[below right] {$E$}-- (-1+10+2,4)node[below right] {$F$}-- (1+10,0)node[below right] {$B$}-- (-1+10,0)node[below right] {$D$}-- (-3+10+2,4);
\draw[step=.25cm,gray,thick] (-1+10,2)node[below right] {$G$}-- (1+10,2)node[below right] {$H$};
\draw[step=.25cm,gray,thick] (-1+10,1)node[below right] {$I$};
\draw[step=.25cm,gray,thick] (1+10,1)node[below right] {$J$};
\draw[step=.25cm,gray,thick] (-1+6,4)-- (-1+6,0);
\draw[step=.25cm,gray] (-1,-1)-- (1,-1)-- (1,1)-- (-1,1)-- (-1,-1);
\draw[step=.25cm,gray] (-1+6,-1)-- (1+6,-1)-- (1+6,1)-- (-1+6,1)-- (-1+6,-1);
\draw[step=.25cm,gray] (-1+10,-1)-- (1+10,-1)-- (1+10,1)-- (-1+10,1)-- (-1+10,-1);
\draw[step=.25cm,gray,thick] (0,-2)-- (0,2);
\draw[step=.25cm,gray,thick] (-1,4)node[below right] {$E$};
\draw[step=.25cm,gray,thick,dashed] (1,0)-- (-1,4)-- (-3,4);
\draw[step=.25cm,gray,thick,dashed] (5,4)-- (7,4)node[below right] {$F$};
\draw[step=.25cm,gray,thick,dashed] (7,0)-- (7,4);
\end{tikzpicture}
\caption{$A^\tau \Gamma_0$}
\end{figure}

3.
For any matrix $A$ in the collection
$\left[\begin{array}{cc} -1 & 2 \\  -2 & 2 \end{array}\right]$ and
$\left[\begin{array}{cc} 1 & -2 \\  2 & -2 \end{array}\right]$
their corresponding $A^\tau \Gamma_0$ has the same vertices
$\left\{
\left(\begin{array}{c} 3 \\  -4  \end{array}\right),
\left(\begin{array}{c} -3 \\  4  \end{array}\right),
\left(\begin{array}{c} 1 \\   0  \end{array}\right),
\left(\begin{array}{c} -1 \\  0  \end{array}\right)
\right\}$
and the same vector
$\vec{q}_A=\left(\begin{array}{c} 0 \\  1  \end{array}\right).$
It is enough to discuss only one of the cases.

Let $A=\left[\begin{array}{cc} 1 & -2 \\  2 & -2 \end{array}\right].$
Then by Proposition \ref{Ataugamma}  $A^\tau =
\left[\begin{array}{cc} 1 & 2 \\  -2 & -2 \end{array}\right]$ and $\vec{q}_A=\left(\begin{array}{c} 0 \\  1  \end{array}\right). $
As in Figure 3 (left), $A^\tau \Gamma_0$ is the parallegram $ABCD$.
So we have
\begin{eqnarray*}
A^\tau \Gamma_0
&\stackrel{2}{\sim}& \triangle BCD \dotcup \Big( \triangle ABD +
\left(\begin{array}{c} -2 \\  4  \end{array}\right)\Big)\\
&\stackrel{2}{\sim}& \triangle BED \dotcup
\Big( \triangle DEC +
\left(\begin{array}{c} 2 \\  0  \end{array}\right)\Big)\\
&=& \Box BFED.
\end{eqnarray*}
So we have shown $A^\tau \Gamma_0  \stackrel{2}{\sim}\Box BFED.$

Let $\Gamma_1 \equiv \Box BHGD $ and
$\Gamma_2 \equiv \Box HFEG.$
We have $\Gamma_1 \stackrel{2}{\sim} \Gamma_0 $ since\\
\begin{eqnarray*}
\Box BHGD &=& \Box BJID \dotcup \Box JHGI \\
&\stackrel{2}{\sim}&
\Box BJID \dotcup \Big( \Box JHGI + \left(\begin{array}{c} 0 \\  -2  \end{array}\right) \Big) \\
&=&\Gamma_0.
\end{eqnarray*}
Also, since $\Box HFEG +\left(\begin{array}{c} 0 \\  1  \end{array}\right)
=\Gamma_0 +\left(\begin{array}{c} 0 \\  4  \end{array}\right),$ we have $\Gamma_2 + \vec{q}_A \stackrel{2}{\sim} \Gamma_0.$

\end{proof}

\section{ Lawton's Equations and Filter Function $m_0$}
\bigskip

Through out the rest of this paper, $A$ will be one of the six matrices as stated in list (\ref{newsix}).
Let
$N_0 \in\N$ and $\mathcal{S}=\{h_{\vec{n}}:~{\vec{n}\in\Z^2}\}$ be a complex solution to the following system of equations
\begin{equation}\label{lawtoneq}
\left\{\begin{array}{l}
\sum_{\vec{n}\in\Z^2}h_{\vec{n}}\overline{h_{\vec{n}+\vec{k}}}=\delta_{\vec{0} \vec{k}},~ \vec{k}\in A^\tau\Z^2 \\
\sum_{\vec{n}\in\Z^2}h_{\vec{n}}=\sqrt{2}.
\end{array}\right.
\end{equation}
with the property that  $h_{\vec{n}}=0$ for all
$\vec{n}\in \Z^2 \backslash [-N_0,N_0]^2$. Let us denote
$\Lambda_0 \equiv \Z^2 \cap [-N_0,N_0]^2.$
Here  $\delta$ is the Kronecker's notation. We will call the system of equations
(\ref{lawtoneq})
\textit{Lawton's system of equations for normalized frame wavelets in $2D$}, or \textit{Lawton's equations}.

Define
\begin{equation}\label{m0}
    m_0(\vec{t})=\frac{1}{\sqrt{2}}\sum_{\vec{n}\in\Z^2} h_{\vec{n}} e^{-i\vec{n}\circ\vec{t}}
    =\frac{1}{\sqrt{2}}\sum_{\vec{n}\in\Lambda_0} h_{\vec{n}} e^{-i\vec{n}\circ\vec{t}},\vec{t} \in \C^2.
\end{equation}
This is a finite sum and $m_0(0)=1$. It is a $2\pi$-periodic trigonometric polynomial function in the sense that $m_0(\vec{t})=m_0(\vec{t}+\pi\vec{t}_0),\forall \vec{t}_0\in (2\Z)^2.$

\begin{proposition}\label{filter}
Let $A$ be an expansive $2 \times 2$ integral matrix and $\vec{q}_A  $ is as stated in
Proposition \ref{properties}.
Let $m_0$ be  defined as in (\ref{m0}), then $m_0$  satisfies
\begin{equation}\label{m0eq}
 |m_0(\vec{t})|^2+|m_0(\vec{t}+\pi \vec{q}_A)|^2=1, \ \forall \vec{t}\in\R^2.
 \end{equation}
\end{proposition}
\begin{remark}\label{less1}
By Proposition \ref{filter} we have $|m_0(\vec{t})|\leq 1$ for all $\vec{t}\in\R^2$. Also, (\ref{m0eq}) may not hold for $\vec{t}\in\C^2$ in general.
\end{remark}

\begin{proof}We have\footnote{In the calculation, the infinite sum is always converging since there are only finite many non-zero terms.}
\begin{eqnarray*}
& &|m_0(\vec{t})|^2+|m_0(\vec{t}+\pi\vec{q}_A)|^2
=  \frac{1}{2}\left| \sum_{\vec{m}\in\Z^2} h_{\vec{m}} e^{-i \vec{m}\circ\vec{t}}\right|^2 + \frac{1}{2}\left| \sum_{\vec{m}\in\Z^2} h_{\vec{m}} e^{-i \vec{m}\circ(\vec{t}+\pi\cdot \vec{q}_A)}\right|^2 \\
 &=& \frac{1}{2}\left[ \sum_{\vec{m}\in\Z^2,\vec{n}\in\Z^2} h_{\vec{m}}\overline{h_{\vec{n}}} e^{-i(\vec{m}-\vec{n})\circ\vec{t}} +\sum_{\vec{m}\in\Z^2,\vec{n}\in\Z^2} (-1)^{(\vec{m}-\vec{n})\circ \vec{q}_A} h_{\vec{m}}\overline{h_{\vec{n}}} e^{-i(\vec{m}-\vec{n})\circ\vec{t}}  \right] \\
 &=&
 \frac{1}{2}\left[
 \sum_{\vec{m}\in\Z^2,\vec{k}\in\Z^2}  h_{\vec{m}}\overline{h_{\vec{m}+\vec{k}}} e^{i\vec{k}\circ\vec{t}}+\sum_{\vec{m}\in\Z^2,\vec{k}\in\Z^2} (-1)^{-\vec{k}\circ \vec{q}_A} h_{\vec{m}}\overline{h_{\vec{m}+\vec{k}}} e^{i\vec{k}\circ\vec{t}}
 \right] \\
\end{eqnarray*}
Here  $\vec{n}$ is replaced by $\vec{m}+\vec{k}.$

By Proposition \ref{properties},
$\vec{k}\circ \vec{q}_A$ is odd when $\vec{k}\in(\ell_A+A^\tau \Z^2).$ Terms $(-1)^{-\vec{k}\circ \vec{q}_A} h_{\vec{m}}\overline{h_{\vec{m}+\vec{k}}} e^{i\vec{k}\circ\vec{t}}$ in the second sum cancel terms
$ h_{\vec{m}}\overline{h_{\vec{m}+\vec{k}}} e^{i\vec{k}\circ\vec{t}}$ in the first sum. The term  $\vec{k}\circ \vec{q}_A$ is even when $\vec{k}\in A^\tau \Z^2.$ So by definition of $\{h_{\vec{t}}\}$ we have
\begin{eqnarray*}
&& |m_0(\vec{t})|^2+|m_0(\vec{t}+\pi\vec{q}_A)|^2
 =
 \sum_{\vec{m}\in\Z^2,\vec{k}\in A^\tau\Z^2}  h_{\vec{m}}\overline{h_{\vec{m}+\vec{k}}} e^{i\vec{k}\circ\vec{t}} \\
 &=&
 \sum_{\vec{k}\in A^\tau\Z^2} \left( \sum_{\vec{m}\in\Z^2}
  h_{\vec{m}}\overline{h_{\vec{m}+\vec{k}}} \right) e^{i\vec{k}\circ\vec{t}}
 =
 \sum_{\vec{k}\in A^\tau\Z^2} \delta_{\vec{0} \vec{k}} e^{i\vec{k}\circ\vec{t}}
 = 1.
\end{eqnarray*}
\end{proof}

\section{The Frame Scaling Function $\varphi$}
\bigskip

Define
\begin{equation}\label{g}
 g(\vec{\xi})= \frac{1}{2\pi}\prod_{j=1}^{\infty}m_0((A^{\tau})^{-j} \vec{\xi}), \forall \vec{\xi}\in\R^2 \text{ and}\\
\end{equation}
\begin{equation}\label{phi}
\varphi = \F^{-1} g.
\end{equation}

In this section we will prove that $g$ and $\varphi$ are well defined $L^2(\R^2)$ functions. We will also prove that, in the extended domain $\C^2$, $g$ is an entire function and $\varphi$ has a compact support in $\R^2$. We will call $\varphi$ the \textit{scaling function}.

For $z\in\C,$ define
\begin{equation}\label{v_xi}
v(z)=\left\{
\begin{array}{ll}
\frac{e^z-1}{z} \quad &z\neq0 \\
 1 \quad &z=0.
\end{array} \right.
\end{equation}
The function $v(z)$ is an entire function on $\C.$

We will need the following inequality in the proofs of Lemma \ref{ineq} and Proposition \ref{compactsupport}.
\begin{lemma}\label{eiz}
\begin{equation}
\label{E} |e^{-iz}-1|\leq \min(2,|z|), \forall z\in \C,\mathfrak{Im} (z)\leq 0.
\end{equation}
\end{lemma}
\begin{proof}
Let $z=a+ib,$ with $b=\mathfrak{Im} (z)\leq 0.$
So we have
\begin{equation}
\label{b} |e^{-iz}-1| \leq 1+ |e^{-iz}| \leq 1 + e^b \leq 2.
\end{equation}
Next we will show that
\begin{equation*}
\label{c}  |e^{-iz}-1| \leq |z|, \forall b\leq 0.
\end{equation*}
We have
\begin{eqnarray*}
  |e^{-iz}-1|^2 &=& e^{2b} -2e^b \cos a +1 \\
  &=& (e^b-1)^2 +2e^b (1-\cos a )
\end{eqnarray*}
Since $e^b > 1+b, \forall b\neq 0,$ when $b < 0,$ $b^2=(-b)^2 >(1-e^b)^2.$ Also, $2e^b (1-\cos a )\leq 2 (1-\cos a )=4\sin^2 \frac{a}{2}\leq a^2.$ So,
\begin{eqnarray*}
  |e^{-iz}-1|^2 &\leq&  b^2+a^2 =|z|^2.
\end{eqnarray*}
This proves the inequality.

\end{proof}

\begin{lemma}\label{gconverge}
Let $A$ be an expansive integral matrix with $|\det(A)|=2$, $\Omega$ be a bounded closed region in $\C^2,$ and $d_j(\vx)
\equiv m_0((A^\tau)^{-j}\vx)-1$. 
Then
\begin{equation}
|d_j (\vx)| \leq C_\Omega \|(A^\tau)^{-1}\|^{j},\forall j\in\N,\vx\in \Omega
\end{equation}
for some constant $C_\Omega>0.$
\end{lemma}
\begin{proof}
 By definition of $h_{\vec{n}}$ we have
\begin{align*}
|d_j(\vx)| &= |m_0((A^\tau)^{-j}\vx)-1|\\
           &= \left|\frac{1}{\sqrt{2}}\sum_{\vec{n}\in\Z^2} h_{\vec{n}} e^{-i\vec{n}\circ (A^\tau)^{-j}\vx}-1 \right| \\
           &= \left|\frac{1}{\sqrt{2}}\sum_{\vec{n}\in\Lambda_0} h_{\vec{n}} (e^{-i\vec{n}\circ (A^\tau)^{-j}\vx}-1) \right| \\
           &= \left|\frac{1}{\sqrt{2}}\sum_{\vec{n}\in\Lambda_0} h_{\vec{n}} v(-i\vec{n}\circ(A^\tau)^{-j}\vx)[-i\vec{n}\circ (A^\tau)^{-j}\vx]\right|  \\
           &\leq \frac{1}{\sqrt{2}}\sum_{\vec{n}\in\Lambda_0} |v(-i\vec{n}\circ(A^\tau)^{-j}\vx)|
           \cdot |-i\vec{n}\circ (A^\tau)^{-j}\vx|,
\end{align*}
since $|h_n| \leq 1$ by Remark \ref{less1}.

For $\vec{n}\in\Lambda_0, \vx\in\Omega,$ we have  $|\vx|\leq M_\Omega$ for some $M_\Omega>0,$ and
\begin{align*}
|-i\vec{n}\circ (A^\tau)^{-j}\vx|
&\leq \sqrt{2} N_0 \cdot M_\Omega \cdot \|(A^\tau)^{-1}\|^{j}
= C_1  \|(A^\tau)^{-1}\|^{j} \leq C_1.
\end{align*}
where $C_1\equiv \sqrt{2} N_0 M_\Omega.$
Let $C_2$ be the finite least upper bound for continuous function $|v(z)|, |z|\leq C_1.$

So, we have
\begin{equation*}
|d_j(\vx)| \leq \frac{1}{\sqrt{2}}
           (2N_0+1)^2 C_1C_2  \|(A^\tau)^{-1}\|^{j}.
\end{equation*}
Therefore
\begin{equation}
|d_j(\vx)| \leq C_\Omega \|(A^\tau)^{-1}\|^{j}
\end{equation}
where $C_\Omega \equiv\frac{1}{\sqrt{2}}
           (2N_0+1)^2 C_1C_2$.

\end{proof}

\begin{proposition}\label{entire}
The function $g(\vx)$  is an entire function on $\C^2$.
\end{proposition}
\begin{remark}\label{gbd}
By \eqref{m0eq} and definition of $g$, it is clear that the function $g$ is bounded on $\R^2.$
\end{remark}
\begin{proof}
For $J\in\N,$ define
\begin{equation}
g_J (\vx) = \frac{1}{2\pi}\prod_{j=1}^{J}m_0((A^\tau)^{-j}\vx),  \forall \vx\in\C^2.
\end{equation}
It is clear that $g_J$ is an entire function.
We have

\begin{align*}
g_J (\vx)  & = \frac{1}{2\pi}\prod_{j=1}^{J}m_0((A^\tau)^{-j}\vx)\\
           & = \frac{1}{2\pi}\prod_{j=1}^{J}(1+d_j(\vx))
\end{align*}

By Lemma \ref{gconverge}, $\sum |d_j (\vx)|$ converges uniformly on bounded region $\Omega,$
the product $\prod_{j=0}^\infty (1+|d_j(\vx)|)$ converges uniformly on $\Omega.$
This implies that   $g$ is the uniform limit of a sequence of entire functions $g_J.$ By Morera Theorem
$g$ is an entire function on $\C^2.$
\end{proof}

\begin{proposition}\label{l2phi}
The functions $g$ and $\varphi$ are in $L^2(\R^2).$
\end{proposition}

\begin{proof}
Will use $\Gamma_\pi$ to denote $[-\pi,\pi]^2.$
For $J\in\N,$ we define on $\R^2$
\begin{equation*}M_J(\vx)=\left\{
\begin{array}{cl}
\prod_{j=1}^{J}|m_0((A^\tau)^{-j}\vx)|^2 & \text{if $\vx\in(A^\tau)^{J+1}\Gamma_\pi$,}\\
0 & \text{if $\vx \in \R^2 \backslash (A^\tau)^{J+1}\Gamma_\pi$}.
\end{array}\right.
\end{equation*}
Since $A$ is expansive, $A^\tau$ is expansive. We have $\lim M_J(\vx) = 4\pi^2|g(\vx)|^2, \vx\in\R^2. $ To prove the Proposition, by Fatou's Lemma it suffices to show  that $\{ \int_{\R^2}M_J(\vx)d\vx, J\in\N\}$ is a bounded sequence.

We have
\begin{eqnarray*}
\int_{\R^2}M_J(\vx)d\vx
&=& \int_{(A^\tau)^{J+1}\Gamma_\pi}\prod_{k=1}^{J}|m_0((A^\tau)^{-k}\vx)|^2 d\vx\\
&=& \int_{(A^\tau)^{J}(A^\tau \Gamma_\pi)}|m_0((A^\tau)^{-J}\vx)|^2 \cdot \prod_{k=1}^{J-1}|m_0((A^\tau)^{-k}\vx)|^2 d\vx
\end{eqnarray*}
Use $\vec{\eta}\equiv (A^\tau)^{-J} \vx,$ by Proposition \ref{Ataugamma} we have
\begin{eqnarray*}
\int_{\R^2}M_J(\vx)d\vx
&=& |\det((A^\tau)^{J})| \int_{A^\tau\Gamma_\pi }|m_0(\vec{\eta})|^2 \cdot \prod_{m=1}^{J-1}|m_0((A^\tau)^{m}\eta)|^2  d\vec{\eta}\\
&=&|\det((A^\tau)^{J})| \Big( \int_{\Gamma_\pi }|m_0(\vec{\eta})|^2 \cdot \prod_{m=1}^{J-1}|m_0((A^\tau)^{m}\eta)|^2  d\vec{\eta}\\
&& + \int_{\Gamma_\pi +\pi \vec{q}_A}|m_0(\vec{\eta})|^2 \cdot \prod_{m=1}^{J-1}|m_0((A^\tau)^{m}\eta)|^2  d\vec{\eta} \Big) \\
&=& |\det((A^\tau)^{J})| \Big( \int_{\Gamma_\pi }|m_0(\vec{\eta})|^2 \cdot \prod_{m=1}^{J-1}|m_0((A^\tau)^{m}\eta)|^2  d\vec{\eta} \\
&& +  \int_{\Gamma_\pi}|m_0(\vec{\eta} -\pi \vec{q}_A)|^2 \cdot \prod_{m=1}^{J-1}|m_0((A^\tau)^{m}\eta -\pi (A^\tau)^{m} \vec{q}_A)|^2  d\vec{\eta} \Big)
\end{eqnarray*}
Since $m_0$ is $2\pi$-periodical, and by Proposition \ref{properties}  equation (\ref{2z}) we have $A^\tau \vec{q}_A \in (2\Z)^2,$ so $\pi (A^\tau)^{m} \vec{q}\in\pi (2\Z)^2$. By Corollary \ref{click}  we have
\begin{eqnarray*}
\int_{\R^2}M_J(\vx)d\vx
&=& |\det((A^\tau)^{J})| \Big(   \int_{\Gamma_\pi }|m_0(\vec{\eta})|^2 \cdot \prod_{m=1}^{J-1}|m_0((A^\tau)^{m}\eta)|^2  d\vec{\eta}\\
&& + \int_{\Gamma_\pi}|m_0(\vec{\eta} +\pi \vec{q}_A)|^2 \cdot \prod_{m=1}^{J-1}|m_0((A^\tau)^{m}\eta)|^2  d\vec{\eta} \Big)\\
&=& |\det((A^\tau)^{J})|
 \int_{\Gamma_\pi }(|m_0(\vec{\eta})|^2+|m_0(\vec{\eta}+\pi \vec{q}_A)|^2 )\cdot \prod_{m=1}^{J-1}|m_0((A^\tau)^{m}\eta)|^2  d\vec{\eta}.
\end{eqnarray*}
By equation (\ref{m0eq}) and then using substitution  $\vx \equiv (A^\tau)^J \vec{\eta},$
we obtain
\begin{eqnarray*}
\int_{\R^2}M_J(\vx)d\vx
&=& |\det((A^\tau)^{J})|
 \int_{\Gamma_\pi } \prod_{m=1}^{J-1}|m_0((A^\tau)^{m}\eta)|^2  d\vec{\eta}\\
 &=& \int_{(A^\tau)^{J}\Gamma_\pi}\prod_{k=1}^{J-1}|m_0((A^\tau)^{-k}\vx)|^2 d\vx\\
&=& \int_{\R^2}M_{J-1}(\vx)d\vx.
\end{eqnarray*}
This proves that the sequence $ \int_{\R^2}M_J(\vx)d\vx, J\in\N$ is a constant sequence.
Therefore, $g$ is square integrable on $\R^2.$
By Plancherel Theorem the function $\varphi$ which is the Fourier inverse transform of $g$,   is in $L^2(\R^2)$.

\end{proof}

Next, we will prove that the scaling function $\varphi$
has a compact support in $\R^2$. We will need the following Schwartz's Paley-Wiener Theorem.

\begin{theorem}(Schwartz's Paley-Wiener Theorem)
An entire function $F$ on $\C^d,d\in\N,$ is the Fourier Transform of a distribution with compact support in $\R^d$ if and only if
there are some constants $C,N$ and $B$, such that
\begin{equation}\label{ineqcpt}
|F(\vx)| \leq C(1+|\vx|)^N e^{B |\mathfrak{Im} (\vx)|}, ~ \forall \vx \in\C^d
\end{equation}
The distribution is supported on the closed ball of center $\vec{0}$ and radius $B$.
\end{theorem}
\begin{remark}
In our current situation, $d=2$ and the distribution $\varphi$ is a regular $L^2(\R^2)$ function as we proved in Proposition \ref{l2phi}.
\end{remark}

\begin{lemma}\label{ineq}
There exist constants $B_0, C_0$ such that for all $ j\in\N,  \vx \in \C^2$,
\begin{equation*}
    m_0 \big( (A^\tau)^{-j} \vx \big) \leq
    e^{B_0 \|(A^\tau)^{-1}\|^j |\mathfrak{Im}(\vx)| }
\big( 1+  C_0
\min (1, \|(A^\tau)^{-1}\| ^j | \vx|)\big).
\end{equation*}

\end{lemma}

\begin{proof}
Let $j\in\N, \vx\in \C^2$ and $(A^\tau)^{-j}\vx = \left(\begin{array}{c} \xi_1 \\  \xi_2  \end{array}\right)\in\C^2.$  Define $ \vec{\ell}_{\vx} = \left(\begin{array}{c} \ell_1 \\  \ell_2  \end{array}\right)\in\Z^2$ by
\begin{equation*}
\ell_m =
\left\{
\begin{array}{rl}
-N_0,        & \text{if } \mathfrak{Im}(\xi_m ) \leq 0; \\
N_0,    & \text{if } \mathfrak{Im} (\xi_m ) > 0.
\end{array}
\right.  m=1,2.
\end{equation*}
then
$\mathfrak{Im}\Big((\vec{n}-\vec{\ell}_{\vx}) \circ \big((A^\tau)^{-j} \vx\big)\Big)\leq 0$
for $\vec{n}\in\Lambda_0.$ We have $|(\vec{n} - \vec{\ell}_{\vx})| \leq  2\sqrt{2}N_0, \forall \vec{n}\in\Lambda_0.$ We denote $B_0 \equiv 4\sqrt{2}N_0$.  It is clear that $|\vec{\ell}_\xi| \leq \frac{B_0}{2}$ and $|(\vec{n} - \vec{\ell}_{\vx})| \leq \frac{B_0}{2}.$
By Lemma \ref{eiz} we have
\begin{eqnarray*}
 |e^{-i(\vec{n}-\vec{\ell}_{\vx}) \circ \big((A^\tau)^{-j} \vx\big)} -1|
 & \leq &
 \min (2,  |(\vec{n}-\vec{\ell}_{\vx}) \circ \big((A^\tau)^{-j} \vx\big)|),\forall \vec{n}\in\Lambda_0
\end{eqnarray*}
So, we have
\begin{equation}\label{X}
|e^{-i(\vec{n}-\vec{\ell}_{\vx}) \circ \big((A^\tau)^{-j} \vx\big)} -1|
\leq \min \big(2, B_0 \|(A^\tau)^{-1}\| ^j | \vx|\big) ,\forall \vec{n}\in\Lambda_0.\\
\end{equation}
We also have
\begin{eqnarray*}
 |e^{-i \vec{\ell}_{\vx} \circ \big((A^\tau)^{-j} \vx\big)}|
 &=& e^{\vec{\ell}_{\vx} \circ \big((A^\tau)^{-j} \mathfrak{Im}(\vx) \big)}\\
 &\leq& e^{|\vec{\ell}_{\vx}|  \|(A^\tau)^{-1}\|^j |\mathfrak{Im}(\vx)| }.
 \end{eqnarray*}
This implies
\begin{equation}\label{Y}
 |e^{-i \vec{\ell}_{\vx} \circ \big((A^\tau)^{-j} \vx\big)}|
 \leq e^{B_0  \|(A^\tau)^{-1}\|^j |\mathfrak{Im}(\vx)| }.
\end{equation}
Since
\begin{eqnarray*}
m_0 \big( (A^\tau)^{-j} \vx \big)
            & =&\sum_{\vec{n}\in\Z^2}\frac{1}{\sqrt{2}}h_{\vec{n}}
e^{-i\vec{n}\circ \big((A^\tau)^{-j} \vx\big)} \\
            & =& e^{-i\vec{\ell}_{\vx} \circ \big((A^\tau)^{-j} \vx\big)}
             \sum_{\vec{n}\in\Lambda_0}\frac{1}{\sqrt{2}}h_{\vec{n}}
e^{-i(\vec{n}-\vec{\ell}_{\vx}) \circ \big((A^\tau)^{-j} \vx\big)} \\
            & =& e^{-i\vec{\ell}_{\vx} \circ \big((A^\tau)^{-j} \vx\big)}
          \Big( 1+  \sum_{\vec{n}\in\Lambda_0}\frac{1}{\sqrt{2}}h_{\vec{n}}
\big(e^{-i(\vec{n}-\vec{\ell}_{\vx}) \circ \big((A^\tau)^{-j} \vx\big)} -1\big)\Big).
\end{eqnarray*}
By (\ref{X}) and (\ref{Y}) we obtain
\begin{eqnarray*}
|m_0 \big( (A^\tau)^{-j} \vx \big)|
&\leq& |e^{ -i\vec{\ell}_{\vx} \circ \big((A^\tau)^{-j} \vx\big)}|
          \big( 1+  \sum_{\vec{n}\in\Lambda_0}\frac{1}{\sqrt{2}}|h_{\vec{n}}|\cdot
|(e^{-i(\vec{n}-\vec{\ell}_{\vx}) \circ \big((A^\tau)^{-j} \vx\big)} -1)\big)|\\
&\leq& e^{B_0 \|(A^\tau)^{-1}\|^j |\mathfrak{Im}(\vx)| }
\big( 1+  \frac{1}{\sqrt{2}}(2N_0+1)^2
\min (2, B_0 \|(A^\tau)^{-1}\| ^j | \vx|)\big)\\
&\leq& e^{B_0 \|(A^\tau)^{-1}\|^j |\mathfrak{Im}(\vx)| }
\big( 1+  C_0
\min (1, \|(A^\tau)^{-1}\| ^j | \vx|)\big).
\end{eqnarray*}
where
$C_0 \equiv \max (\sqrt{2}(2N_0+1)^2, \frac{B_0}{\sqrt{2}}(2N_0+1)^2 ).$

\end{proof}

\begin{proposition}\label{compactsupport}
The scaling function $\varphi$ is an $L^2 (\R^2)$ function with compact support.
\end{proposition}

\begin{proof}
Let $\vx\in\R^2, \vx\neq \vec{0}.$\footnote{The case when $\vx = \vec{0}$ is trivial. We omit it.}
By Schwartz's Paley-Wiener Theorem, it suffices to prove that the function $g$ satisfies the inequality
(\ref{ineqcpt}).

We write $\beta=\|(A^\tau)^{-1}\|^{-1}.$ Since $A$ is expansive, $\beta\in (1,\infty)$.
 We have
 \begin{equation*}
   \prod_{j=1}^{\infty} e^{B_0 \|(A^\tau)^{-1}\|^j |\mathfrak{Im}(\vx)| }
   =e^{B |\mathfrak{Im}(\vx)|}
 \end{equation*}
 where $B \equiv \sum_{j=1} ^\infty \frac{B_0}{\beta^j}.$ So by Lemma \ref{ineqcpt} we have
 \begin{eqnarray*}
  |g(\vx)|
  &=& \big| \frac{1}{2\pi}\prod_{j=1}^{\infty}m_0((A^{\tau})^{-j} \vec{\xi})\big|\\
  &\leq&
  \frac{1}{2\pi}\prod_{j=1}^{\infty} e^{B_0 \|(A^\tau)^{-1}\|^j |\mathfrak{Im}(\vx)| }
\big( 1+  C_0
\min (1, \|(A^\tau)^{-1}\| ^j | \vx|)\big)\\
  &\leq& \frac{1}{2\pi} e^{B|\mathfrak{Im}(\vx)| }
\prod_{j=1}^{\infty}
  \big( 1+  C_0\min (1, \|(A^\tau)^{-1}\| ^j | \vx|)\big)\\
    &=& \frac{1}{2\pi} e^{B|\mathfrak{Im}(\vx)| }
\prod_{j=1}^{\infty}
  \big( 1+  C_0\min (1, \frac{|\vx|}{\beta^j})\big).
\end{eqnarray*}
On the other hand, the sequence $\{\beta^{j}\}$ is monotonically increasing to $+\infty$.
 Let $I_j \equiv [\beta^j,\beta^{j+1}), j\in\N$ and $I_0 \equiv (0,\beta).$
The set of intervals $\{ I_j, j\geq 0 \}$ is a partition of $(0,\infty).$
So $|\vx| \in I_{j_0}$ for some integer $j_0 \geq 0.$
We have
\begin{eqnarray*}
  (1+C_0)^{j_0} &=& \beta^{j_0 \log _{\beta} (1+C_0)}\\
&\leq& |\vx|^{\log _{\beta} (1+C_0)}\\
&\leq& (1+|\vx|)^N,
\end{eqnarray*}
where $N $ is the smallest natural number no less than $ \log _{\beta} (1+C_0).$ This is a constant related to $A$ and $N_0$ only.
So, we have
 \begin{eqnarray*}
  |g(\vx)|
  &\leq& \frac{1}{2\pi} (1+C_0)^{j_0} e^{B|\mathfrak{Im}(\vx)| }
\prod_{j=j_0+1}^{\infty}
  \big( 1+  C_0\min (1, \frac{| \vx|}{\beta^j})\big)\\
  &\leq& (1+|\vx|)^N e^{B|\mathfrak{Im}(\vx)| }\cdot \Big(\frac{1}{2\pi}
\prod_{j=j_0+1}^{\infty}
  \big( 1+  C_0\min (1, \frac{| \vx|}{\beta^j})\big)\Big).
\end{eqnarray*}
Now, since $|\vx|\in I_{j_0}=[\beta^{j_0},\beta^{j_0+1}), \frac{|\vx|}{\beta^{j_0+1}} < 1$.
We have
\begin{eqnarray*}
  \frac{1}{2\pi}
\prod_{j=j_0+1}^{\infty}
  \big( 1+  C_0\min (1, \frac{| \vx|}{\beta^j})\big)
&=&
  \frac{1}{2\pi} \prod_{j=j_0+1}^{\infty}
  (1+C_0 \frac{| \vx|}{\beta^{j_0+1}} \cdot \frac{1}{\beta^{j-(j_0+1)}})\\
&\leq&
  \frac{1}{2\pi} \prod_{k=0}^{\infty}
  (1+   \frac{C_0}{\beta^{k}})\\
&\leq&
  \frac{1}{2\pi}
  e^{\sum \frac{C_0}{\beta^{k}}}.
\end{eqnarray*}
Denote $C \equiv \frac{1}{2\pi}
  e^{\sum \frac{C_0}{\beta^{k}}}.$
This is a constant decided by the matrix $A$.

Combining the above argument, we have
\begin{equation*}
|g(\vx)| \leq C (1+|\vx|)^N e^{B|\mathfrak{Im}(\vx)| }.
\end{equation*}
\end{proof}

\section{Normalized Tight Frame Wavelet Function $\psi$}

In this section we will construct a normalized tight frame wavelet function $\psi$ associated with the scaling function $\varphi.$
By definition  \eqref{g} and Lemma \ref{k} we have
\begin{eqnarray*}
  \widehat{\varphi}(\vs)
  &=&g(\vs) =  m_0 ((A^\tau)^{-1} \vs) \cdot \frac{1}{2\pi}\prod_{j=2}^{\infty}m_0((A^{\tau})^{-j} \vec{\xi})\\
  &=& m_0 ((A^\tau)^{-1} \vs) g((A^\tau)^{-1} \vs)\\
  &=&
  \frac{1}{\sqrt{2}}\sum_{\vec{n}\in\Lambda_0} h_{\vec{n}} e^{-i\vec{n}\circ (A^\tau)^{-1}\vs} g((A^\tau)^{-1} \vs)\\
  &=&
  \sum_{\vec{n}\in\Lambda_0} h_{\vec{n}} \widehat{T}_{A^{-1}\vec{n}}\widehat{D}_A g(\vs)\\
  &=&
  \sum_{\vec{n}\in\Lambda_0} h_{\vec{n}} \widehat{D}_A\widehat{T}_{\vec{n}} \widehat{\varphi}(\vs).
\end{eqnarray*}
Taking  Fourier inverse transform on two sides, we have
\begin{eqnarray}\label{2rel}
  \varphi &=& \sum_{\vec{n}\in\Lambda_0} h_{\vec{n}} D_AT_{\vec{n}} \varphi , \text{ or}\\
    \varphi (\vec{t}) &=& \sqrt{2} \sum_{\vec{n}\in\Lambda_0} h_{\vec{n}} \varphi (A\vec{t}-\vec{n}), \ \vt\in\R^2.
\end{eqnarray}

Define
\begin{equation*}
    \sigma_A (\vec{n}) = \left\{
    \begin{array}{cccc}
      0 & \vec{n} \in A \Z^2,  \\
      1 & \vec{n} \notin A \Z^2.
    \end{array}
    \right.
\end{equation*}
\begin{remark}\label{sigmaA}
It is clear that we have $\sigma_A (\vec{u}+A \vec{v}) = \sigma_A (\vec{u}), \forall \vec{u},\vec{v}\in\Z^2$.
By Proposition \ref{properties} (\ref{aatau}) $A\Z^2 = A^\tau \Z^2,$  so we have $\sigma_A(\vec{n}) =0$ if and only if
$\vec{n} \in A^\tau \Z^2.$ Furthermore, we have $\sigma_A(\vec{\ell}_A)=1$ and $\sigma_A(\vec{\ell}_A-\vec{n})=1-\sigma_A(\vec{n}), \forall \vec{n}\in\Z^2.$
\end{remark}
\begin{definition}\label{psi}
Define a function $\psi$ on $\R^2$ by
\begin{eqnarray}\label{defpsi}
  \psi &=& \sum_{\vec{n}\in\Z^2}
  (-1)^{\sigma_A (\vec{n})} \overline{h_{\vec{\ell}_A-\vec{n}}} D_AT_{\vec{n}} \varphi , \text{ or equivalently }\\
    \psi (\vec{t}) &=& \sqrt{2} \sum_{\vec{n}\in\Z^2}
     (-1)^{\sigma_A (\vec{n})} \overline{h_{\vec{\ell}_A-\vec{n}}}  \varphi (A\vec{t}-\vec{n}),\forall \vec{t}\in\R^2.
\end{eqnarray}
\end{definition}
In this section we will prove that the function $\psi$ is a normalized tight frame wavelet associated with the expansive matrix $A.$ It is clear that the function $\psi$ has a compact support since the scaling function $\varphi$ has a compact support and the sum in the definition for $\psi$ has only finite non-zero terms. For $J\in\Z,$ and $f\in L^2(\R^2)$ define

\begin{eqnarray*}
I_J
 &\equiv&
\sum_{\vec{k}\in\Z^2} \langle f,D_A^JT_{\vec{k}}\varphi\rangle D_A^JT_{\vec{k}}\varphi; \\
F_J
&\equiv&
\sum_{\vec{k}\in\Z^2} \langle f,D_A^JT_{\vec{k}}\psi\rangle D_A^JT_{\vec{k}}\psi.
\end{eqnarray*}

\begin{lemma}\label{telleskope}
Let $f\in L^2(\R^2)$. Then
\begin{equation}\label{T}
  I_{J+1} = I_J + F_J, \forall J\in\Z.
\end{equation}
\end{lemma}

\begin{proof}

1. The case $J=0.$
By the equation \eqref{2rel}, Definition \ref{psi} and Lemma \ref{k}, we have
\begin{eqnarray*}
  I_0 &=&
  \sum_{\vec{k}\in\Z^2} \langle f, T_{\vec{k}}\varphi \rangle T_{\vec{k}}\varphi \\
  &=&
  \sum_{\vec{k}\in\Z^2} \langle f, T_{\vec{k}} \sum_{\vec{p}\in\Z^2} h_{\vec{p}} D_AT_{\vec{p}}\varphi \rangle T_{\vec{k}} \sum_{\vec{q}\in\Z^2} h_{\vec{q}} D_AT_{\vec{q}}\varphi \\
    &=&\sum_{\vec{p}\in\Z^2} \sum_{\vec{q}\in\Z^2} \sum_{\vec{k}\in\Z^2}
 \overline{h_{\vec{p}}} h_{\vec{q}}\langle f,  D_AT_{\vec{p}+A\vec{k}}\varphi \rangle   D_AT_{\vec{q}+A\vec{k}}\varphi \\
 D_AT_{\vec{m}}\varphi \rangle   D_AT_{\vec{n}}\varphi\\
 F_0
& =& \sum_{\vec{k}\in\Z^2} \langle f, T_{\vec{k}}\psi \rangle T_{\vec{k}}\psi \\
& =& \sum_{\vec{k}\in\Z^2} \langle f, T_{\vec{k}} \sum_{\vec{p}\in\Z^2} (-1)^{\sigma_A(\vec{p})} \overline{h_{\vec{\ell}_A-\vec{p}}} D_AT_{\vec{p}}\varphi \rangle T_{\vec{k}} \sum_{\vec{q}\in\Z^2} (-1)^{\sigma_A(\vec{q})} \overline{h_{\vec{\ell}_A-\vec{q}}} D_AT_{\vec{q}}\varphi  \\
& =& \sum_{\vec{p}\in\Z^2}\sum_{\vec{q}\in\Z^2}\sum_{\vec{k}\in\Z^2} (-1)^{\sigma_A(\vec{p})+\sigma_A(\vec{q})}h_{\vec{\ell}_A-\vec{p}}
\overline{h_{\vec{\ell}_A-\vec{q}}}
\langle f,  D_AT_{\vec{p}+A\vec{k}}\varphi \rangle   D_AT_{\vec{q}+A\vec{k}}\varphi
\end{eqnarray*}
Using substitutions $\vec{m} \equiv \vec{p}+ A\vec{k}$ and $\vec{n} \equiv \vec{q}+ A\vec{k},$
we have by Remark \ref{sigmaA}

\begin{eqnarray*}
I_0      &=& \sum_{\vec{m},\vec{n}\in\Z^2}  \sum_{\vec{k}\in\Z^2}
 \overline{h_{\vec{m}-A\vec{k}}} h_{\vec{n}-A\vec{k}}\langle f,  D_AT_{\vec{m}}\varphi \rangle   D_AT_{\vec{n}}\varphi\\
F_0 &=&\sum_{\vec{m},\vec{n}\in\Z^2}  \sum_{\vec{k}\in\Z^2}
(-1)^{\sigma_A(\vec{m}-A\vec{k})+\sigma_A(\vec{n}-A\vec{k})}
h_{\vec{\ell}_A-\vec{m}+A\vec{k}} \overline{h_{\vec{\ell}_A-\vec{n}+A\vec{k}}}
\langle f,  D_AT_{\vec{m}}\varphi \rangle   D_AT_{\vec{n}}\varphi\\
&=&\sum_{\vec{m},\vec{n}\in\Z^2}  \sum_{\vec{k}\in\Z^2}
(-1)^{\sigma_A(\vec{m})+\sigma_A(\vec{n})}
h_{\vec{\ell}_A-\vec{m}+A\vec{k}} \overline{h_{\vec{\ell}_A-\vec{n}+A\vec{k}}}
\langle f,  D_AT_{\vec{m}}\varphi \rangle   D_AT_{\vec{n}}\varphi
\end{eqnarray*}
We will use notations
\begin{eqnarray*}
    \alpha_{\vec{m},\vec{n}} &\equiv& \sum_{\vec{k}\in\Z^2}
 \overline{h_{\vec{m}-A\vec{k}}} h_{\vec{n}-A\vec{k}}=\sum_{\vec{\ell}\in A\Z^2}
 \overline{h_{\vec{m}+\vec{\ell}}} h_{\vec{n}+\vec{\ell}}=\sum_{\vec{\ell}\in \vec{n}+ A\Z^2}
 \overline{h_{\vec{\ell}+(\vec{m}-\vec{n})}} h_{\vec{\ell}}, \\
\beta_{\vec{m},\vec{n}}
&\equiv&
\sum_{\vec{k}\in\Z^2}
(-1)^{\sigma_A(\vec{m})+\sigma_A(\vec{n})}
h_{\vec{\ell}_A-\vec{m}+A\vec{k}} \overline{h_{\vec{\ell}_A-\vec{n}+A\vec{k}}}\\
&=&
 (-1)^{\sigma_A(\vec{m})+\sigma_A(\vec{n})}
 \sum_{\vec{\ell}\in \vec{\ell}_A-\vec{m}+A\Z^2}
\overline{h_{\vec{\ell}+(\vec{m}-\vec{n})}}h_{\vec{\ell}}.
\end{eqnarray*}

If $\vec{m}-\vec{n} \in A\Z^2,$ then $\sigma(\vec{m}) = \sigma_A(\vec{n}), $
$(-1)^{\sigma_A(\vec{m})+\sigma_A(\vec{n})}=1$ and
$(\vec{\ell}_A-\vec{m}+A^\tau\Z^2)\dotcup (\vec{n}+ A^{\tau}\Z^2)=\Z^2.$ By Lawton's equations \eqref{lawtoneq}, we have
\begin{eqnarray*}
\alpha_{\vec{m},\vec{n}}+\beta_{\vec{m},\vec{n}}
&=&
\sum_{\vec{\ell}\in \vec{n}+ A\Z^2}
 \overline{h_{\vec{\ell}+(\vec{m}-\vec{n})}} h_{\vec{\ell}}
+
 \sum_{\vec{\ell}\in \vec{\ell}_A-\vec{m}+A\Z^2}
\overline{h_{\vec{\ell}+(\vec{m}-\vec{n})}}h_{\vec{\ell}}\\
&=&
 \sum_{\vec{\ell}\in\Z^2}
\overline{h_{\vec{\ell}+(\vec{m}-\vec{n})}}h_{\vec{\ell}} = \delta_{\vec{m},\vec{n}}.
\end{eqnarray*}

If $\vec{m}-\vec{n} \in \vec{\ell}_A+A \Z^2,$
then exactly one element of $\vec{m}$ and $\vec{n}$ is in $A\Z^2$ and the other one is
in $\vec{\ell}_A + A\Z^2.$
Then
$(-1)^{\sigma_A(\vec{m})+\sigma_A(\vec{n})}=-1$ and
$(\vec{\ell}_A-\vec{m}+A^\tau\Z^2)= (\vec{n}+ A^{\tau}\Z^2).$
Hence,
\begin{eqnarray*}
\alpha_{\vec{m},\vec{n}}+\beta_{\vec{m},\vec{n}}
&=&
\sum_{\vec{\ell}\in \vec{n}+ A^{\tau}\Z^2}
 \overline{h_{\vec{\ell}+(\vec{m}-\vec{n})}} h_{\vec{\ell}}
\ -
 \sum_{\vec{\ell}\in \vec{\ell}_A-\vec{m}+A^\tau\Z^2}
\overline{h_{\vec{\ell}+(\vec{m}-\vec{n})}}h_{\vec{\ell}}\\
&=&
0 = \delta_{\vec{m},\vec{n}}.
\end{eqnarray*}
Therefore, we have
\begin{eqnarray*}
I_0 + F_0
&=&
\sum_{\vec{m},\vec{n}\in\Z^2}
(\alpha_{\vec{m},\vec{n}}+\beta_{\vec{m},\vec{n}})
\langle f,  D_AT_{\vec{m}}\varphi \rangle   D_AT_{\vec{n}}\varphi\\
&=&
\sum_{\vec{m},\vec{n}\in\Z^2}
\delta_{\vec{m},\vec{n}}
\langle f,  D_AT_{\vec{m}}\varphi \rangle   D_AT_{\vec{n}}\varphi\\
&=&
\sum_{\vec{k}\in\Z^2}
\langle f,  D_AT_{\vec{k}}\varphi \rangle   D_AT_{\vec{k}}\varphi\\
&=& I_1.
\end{eqnarray*}
This is
\begin{equation}\label{A}
\sum_{\vec{k}\in\Z^2} \langle f,D_AT_{\vec{k}}\varphi\rangle D_AT_{\vec{k}}\varphi =
\sum_{\vec{k}\in\Z^2} \langle f,T_{\vec{k}}\varphi\rangle T_{\vec{k}}\varphi
+ \sum_{\vec{k}\in\Z^2} \langle f,T_{\vec{k}}\psi\rangle T_{\vec{k}}\psi.
\end{equation}

2. The general case. Let $f\in L^2(\R^2).$
We  replace $f$ by $ (D_A ^*)^J f$ in equation (\ref{A}) where $D_A ^*$ is the unitary operator dual to $D_A$. Then we have
\begin{equation*}
\sum_{\vec{k}\in\Z^2} \langle (D_A ^J)^* f,D_A T_{\vec{k}}\varphi\rangle D_A T_{\vec{k}}\varphi
=
\sum_{\vec{k}\in\Z^2} \langle (D_A ^J)^* f,T_{\vec{k}}\varphi\rangle T_{\vec{k}}\varphi
+
\sum_{\vec{k}\in\Z^2} \langle (D_A ^J)^* f,T_{\vec{k}}\psi\rangle T_{\vec{k}}\psi.
\end{equation*}
Apply $D_A ^J$ to both sides of the equation. By using
$\langle (D_A ^J)^* f, h \rangle = \langle f, D_A ^J h \rangle,$ we obtain the desired general equality (\ref{T}).

\end{proof}

In the rest of this section, we will establish the main result of this paper. We state Theorem \ref{theom_frame} first. We will complete the proof through lemmas and propositions.
For $f\in L^2 (\R^2)$ and $J\in \Z,$ we will use the following notations.
\begin{eqnarray*}
  L_J(f) &\equiv& \sum_{\vec{\ell} \in \Z^2} |\langle f,D_A^JT_{\vec{\ell}}\varphi \rangle |^2; \text{ in particular}\\
  L_0(f) &=& \sum_{\vec{\ell} \in \Z^2} |\langle f,T_{\vec{\ell}}\varphi \rangle |^2.
\end{eqnarray*}

For a positive number $\rho$ we define functions $f_\rho$ and
$f_{\overline{\rho}}$ by
$\widehat{f_\rho} \equiv \widehat{f} \cdot \chi_{\{|\vt|\leq \rho\}}$
and
$\widehat{f_{\overline{\rho}}} \equiv \widehat{f} \cdot \chi_{\{|\vt| > \rho\}},$
respectively. Here $\chi$ is the characteristic function.
Then we have
$f=f_\rho+f_{\overline{\rho}}.$ Also, it is clear that $\|f\|^2 = \|\widehat{f}\|^2=\|f_\rho\|^2 +\|f_{\overline{\rho}}\|^2$, $\lim_{\rho\rightarrow\infty}\|f_\rho\|^2=\|f\|^2$ and
$\lim_{\rho\rightarrow\infty}\|f_{\overline{\rho}}\|^2=0.$

\begin{theorem}\label{theom_frame}
Let $\psi$ be as defined in Definition \ref{psi}. Then,  $\{D_A ^n T_{\vec{\ell}}\psi, n\in\Z,\vec{\ell}\in\Z^2\}$ is a normalized tight frame for $\L^2(\R^2)$.
\end{theorem}

\begin{proof}
Let $f\in L^2(\R^2).$ We will prove that
\begin{equation}\label{psieq}
f = \sum_{n\in\Z}\sum_{\vec{\ell}\in\Z^2}\langle f, D_A^n T_{\vec{\ell}}\psi\rangle D_A ^n T_{\vec{\ell}}\psi,
\end{equation}
the convergence is in $L^2 (\R^2)$-norm.

By Lemma \ref{telleskope},
we have
$I_j - I_{j-1} =  F_{j-1}$, $\forall j\in\Z$.
Hence
\begin{equation*}
  \sum_{j=-J+1}^{J}F_{j} = I_J - I_{-J}, \forall J\in\Z.
\end{equation*}
This implies that
\begin{eqnarray*}
&&  \sum_{j=-J+1} ^J \sum_{\vec{\ell}\in\Z^2}\langle f, D_A^j T_{\vec{\ell}}\psi\rangle D_A ^j T_{\vec{\ell}}\psi\\
&=&\sum_{\vec{\ell}\in\Z^2}\langle f, D_A^JT_{\vec{\ell}}\varphi\rangle D_A^JT_{\vec{\ell}}\varphi  - \sum_{\vec{\ell}\in\Z^2}\langle f, D_A^{-J}T_{\vec{\ell}}\varphi\rangle D^{-J}T_{\vec{\ell}}\varphi.
\end{eqnarray*}
 Taking inner product of $f$  with both sides of the equation, we have
\begin{eqnarray*}
\sum_{j=-J+1} ^J \sum_{\vec{\ell}\in\Z^2}
|\langle f, D_A^j T_{\vec{\ell}}\psi\rangle|^2
&=& L_J(f) - L_{-J}(f).
\end{eqnarray*}
By Proposition \ref{go0} and Proposition \ref{isnorm}, we have

\begin{align*}
\lim_{J\rightarrow+\infty} & L_J(f)  =  \| f \|^2; \\
\lim_{J\rightarrow+\infty} & L_{-J}(f)  =  0.\label{1}
\end{align*}
So, we have
\begin{equation*}
\sum_{j\in \Z} \sum_{\vec{\ell}\in\Z^2}|\langle f, D_A^j T_{\vec{\ell}}\psi\rangle|^2
=\|f\|^2, \forall f\in L^2(\R^2).
\end{equation*}

\end{proof}

To complete the proof of Theorem \ref{theom_frame}, we will prove Propositions \ref{go0} and
\ref{isnorm} and related Lemmas.
we first need the following

\begin{lemma}\label{conv}
Let $f\in L^2(\R^2)$.  Then
\begin{eqnarray}
\label{x}   L_J(f) = \sum_{\vec{\ell} \in \Z^2} |\langle f,D_A ^J T_{\vec{\ell}}\varphi \rangle |^2
  &\leq&
  (2B+1)^2 \|\varphi \|^2 \| f \|^2, \forall J\in\Z;\\
\label{y} \lim_{\rho\rightarrow\infty}\limsup_{J\rightarrow + \infty} L_J(f_{\overline{\rho}}) &=& 0.
\end{eqnarray}
\end{lemma}
\begin{proof}
By Proposition \ref{compactsupport} the scaling function $\varphi$ has a compact support. Let $B$ be a natural number such that the set $[-B,B)^2$ contains the support of $\varphi.$
We will write
$E_0 \equiv [-\frac{1}{2},\frac{1}{2})^2$,$E_B \equiv [-B-\frac{1}{2},B+\frac{1}{2})^2$
and $\Lambda_B \equiv \Z^2 \cap [-B,B]^2$. For $\vec{n}\in\Z^2,$ we have $\vec{n} = (2B+1)\vec{\ell}+\vec{d},\vec{\ell}\in\Z^2,\vec{d}\in\Lambda_B.$ Here $\vec{\ell}$ and $\vec{d}\in\Lambda_B$ are uniquely determined  by $\vec{n}.$ We have
\begin{eqnarray*}
  \Z^2 &=& \bigcup _{d\in\Lambda_B} \bigcup _{\vec{\ell}\in\Z^2}(2B+1)\vec{\ell}+\vec{d}
\end{eqnarray*}
This is a disjoint union.
Also, $\{E_B+(2B+1)\vec{\ell},\vec{\ell}\in\Z^2\}$ is a partition of $\R^2$. Hence for a fixed $\vec{d} \in \Lambda_B,$
$\{E_B+(2B+1)\vec{\ell}+\vec{d},\vec{\ell}\in\Z^2\}$ is a partition of $\R^2.$
Note that the set $E_B+(2B+1)\vec{\ell}+\vec{d}$ contains the support for
$T_{\vec{n}}\varphi,$ where $\vec{n}=(2B+1)\vec{\ell}+\vec{d}.$
So, for a fixed $\vec{d} \in \Lambda_B,$ supports of functions in the set $\{T_{\vec{n}}\varphi, \vec{n}=(2B+1)\vec{\ell}+\vec{d}, \vec{\ell}\in\Z^2\}$ are disjoint.
Then we have
\begin{eqnarray*}
L_0(f)
&=&
\sum_{\vec{d} \in \Lambda_B}\sum_{\vec{\ell} \in \Z^2} |\langle f,T_{(2B+1)\vec{\ell}+\vec{d}}\varphi \rangle |^2\\
&=&
\sum_{\vec{d} \in \Lambda_B}\sum_{\vec{\ell} \in \Z^2}
\left|\int_{\R^2} \chi_{E_B+(2B+1)\vec{\ell}+\vec{d}}(\vec{t}) f(\vec{t})  T_{(2B+1)\vec{\ell}+\vec{d}}\varphi (\vec{t})d \mu \right|^2 \\
&\leq&
\sum_{\vec{d} \in \Lambda_B}
\left(\|\varphi \|^2\sum_{\vec{\ell} \in \Z^2}
\int_{E_B+(2B+1)\vec{\ell}+\vec{d}} |f(\vec{t})|^2 d\mu \right) \\
&\leq&
\sum_{\vec{d} \in \Lambda_B}
\|\varphi \|^2 \| f \|^2 \leq (2B+1)^2 \|\varphi \|^2 \| f \|^2.  \\
\end{eqnarray*}
So we have
\begin{eqnarray*}
  L_J(f)
  &=&
    \sum_{\vec{\ell} \in \Z^2} |\langle (D_A ^J )^*f, T_{\vec{\ell}}\varphi \rangle |^2
   \leq
 (2B+1)^2 \|\varphi \|^2 \| (D_A ^J )^*f \|^2\\
 &=&(2B+1)^2 \|\varphi \|^2 \| f \|^2.
\end{eqnarray*}
So we have \eqref{x}.
Since $\lim_{\rho\rightarrow\infty} \|f_{\bar{\rho}}\|=0,$ the equality \eqref{y} is an immediate consequence of the inequality \eqref{x} just proved.

\end{proof}

\begin{proposition}\label{go0}
Let  $f\in \L^2(\R^2)$. Then
\begin{equation*}
\lim_{J\rightarrow+\infty}
L_{-J}(f)=0.
\end{equation*}
\end{proposition}

\begin{proof}
Let $f\in L^2(\R^2).$
We have
\begin{eqnarray*}
L_{-J}(f)&=& \sum_{\vec{\ell} \in \Z^2} |\langle  f,D_A^{-J}T_{\vec{\ell} }\varphi \rangle |^2\\
&=&
\sum_{\vec{d} \in \Lambda_B} \sum_{\vec{\ell} \in \Z^2}
 |\langle  f,D_A^{-J}T_{(2B+1)\vec{\ell}+\vec{d} }\varphi \rangle |^2\\
&=&
\sum_{\vec{d} \in \Lambda_B} \sum_{\vec{\ell} \in \Z^2\backslash \{\vec{0}\}}
 |\langle  f,D_A^{-J}T_{(2B+1)\vec{\ell}+\vec{d} }\varphi \rangle |^2+
\sum_{\vec{d} \in \Lambda_B}
 |\langle  f,D_A^{-J}T_{\vec{d} }\varphi \rangle |^2.
\end{eqnarray*}
For each $\vec{d}\in\Lambda_B,$ $\{E_B+(2B+1)\vec{\ell}+\vec{d},\vec{\ell}\in\Z^2\}$
is a partition of $\Z^2.$
It is clear that $E_0 \subset E_B+\vec{d}$ and
$(E_B+(2B+1)\vec{\ell}+\vec{d})\cap E_0 = \emptyset, \forall \vec{\ell} \in \Z^2\backslash \{\vec{0}\}$.
The support of the function $D_A^{-J}T_{\vec{\ell} }\varphi $ is contained in
$A^J (E_B+\vec{\ell}).$ We have
\begin{eqnarray*}
&&\sum_{\vec{\ell} \in \Z^2\backslash \{\vec{0}\}}
 |\langle  f,D_A^{-J}T_{(2B+1)\vec{\ell}+\vec{d} }\varphi \rangle |^2 \\
 =&&
 \sum_{\vec{\ell} \in \Z^2\backslash \{\vec{0}\}}
\left| \int_{\R^2}
\chi_{A^J(E_B+(2B+1)\vec{\ell}+\vec{d}) } \cdot f \cdot \overline{D_A^{-J}T_{(2B+1)\vec{\ell}+\vec{d} }\varphi} d \mu \right|^2\\
 \leq&&
 \sum_{\vec{\ell} \in \Z^2\backslash \{\vec{0}\}}
\int_{A^J(E_B+(2B+1)\vec{\ell}+\vec{d}) } \left| f \right|^2 d \mu \cdot
\|D_A^{-J}T_{(2B+1)\vec{\ell}+\vec{d} }\varphi\|^2\\
\leq&&
\int_{\R^2\backslash A^J E_0}  \left| f \right|^2 d \mu \cdot
\|\varphi\|^2.
\end{eqnarray*}
Since $A$ is expansive, $\lim _{J\rightarrow+\infty}A^J E_0 =\R^2,$ $\lim_{J\rightarrow+\infty} \int_{\R^2\backslash A^J E_0}  \left| f \right|^2 d \mu =0.$
So
\begin{equation*}
\lim_{J\rightarrow+\infty} \sum_{\vec{d} \in \Lambda_B} \sum_{\vec{\ell} \in \Z^2\backslash \{\vec{0}\}}
 |\langle  f,D_A^{-J}T_{(2B+1)\vec{\ell}+\vec{d} }\varphi \rangle |^2
\leq
(2B+1)^2 \|\varphi\|^2\lim_{J\rightarrow+\infty}\int_{\R^2\backslash A^J E_0}  \left| f \right|^2 d \mu=0.
\end{equation*}

To complete the proof of this Proposition, we need to show that
\begin{equation}
\lim_{J\rightarrow+\infty} \sum_{\vec{d} \in \Lambda_B}
 |\langle  f,D_A^{-J}T_{\vec{d} }\varphi \rangle |^2 =0.
\end{equation}
Let $f_N\equiv \chi_{[-N,N]^2} \cdot f.$ Let $\varepsilon>0,$ and choose $N\in\N$ be large
such that $\|f-f_N\|\leq \frac{\varepsilon}{2\|\varphi\|}.$ Then we have
$|\langle  f,D_A^{-J}T_{\vec{d} }\varphi \rangle|\leq |\langle  f_N,D_A^{-J}T_{\vec{d} }\varphi \rangle|+\frac{\varepsilon}{2}.$ Since
\begin{eqnarray*}
|\langle  f_N,D_A^{-J}T_{\vec{d} }\varphi \rangle | &=& |\langle  D_A^{J}f_N,T_{\vec{d} }\varphi \rangle | \\
&=& |\langle  \chi_{A^{-J} [-N,N]^2} D_A^{J}f_N,T_{\vec{d} }\varphi \rangle | \\
&=& |\langle  D_A^{J}f_N,\chi_{A^{-J} [-N,N]^2} T_{\vec{d} }\varphi \rangle |,
\end{eqnarray*}
we have
\begin{eqnarray*}
|\langle  f_N,D_A^{-J}T_{\vec{d} }\varphi \rangle | \rangle |^2
&\leq&
\|D_A^{J}f_N\| \cdot \sqrt{\int_{\R^2}\left|\chi_{A^{-J} [-N,N]^2} T_{\vec{d} }\varphi\right|^2d\mu}\\
&\leq&
\|f\| \cdot \sqrt{\int_{\R^2}\left|\chi_{A^{-J} [-N,N]^2} T_{\vec{d} }\varphi\right|^2d\mu}\cdot \|T_{\vec{d}} \varphi\|\\
&=& \frac{(2N+1)\|f\| \|\varphi\|}{2^{\frac{J}{2}}}.
\end{eqnarray*}
When $J> 2\log_2 \frac{2(2N+1)\|f\| \|\varphi \|}{\varepsilon},$
we have
$\frac{(2N+1)\|f\| \|\varphi\|}{2^{\frac{J}{2}}} < \frac{\varepsilon}{2}$ and
$|\langle  f,D_A^{-J}T_{\vec{d} }\varphi \rangle |<\varepsilon.$
So $\lim_{J\rightarrow+\infty}|\langle  f,D_A^{-J}T_{\vec{d} }\varphi \rangle |^2 =0$ for each $\vec{d}\in \Lambda_B.$ Since $\Lambda_B$ is a finite set, we have
\begin{equation*}
    \lim_{J\rightarrow+\infty} \sum_{\vec{d} \in \Lambda_B}
 |\langle  f,D_A^{-J}T_{\vec{d} }\varphi \rangle |^2 =0.
\end{equation*}

\end{proof}

\begin{lemma}\label{identitylj}
Let $f\in L^2(\R^2)$ and $J\in\Z.$ Then
\begin{equation*}
L_J(f) = (2\pi)^2
\int_{\R^2} \sum_{\vec{\ell}\in\Z^2}  \Big( \widehat{f}(\vt)\overline{\widehat{f}(\vt-2\pi (A^\tau)^J\vec{\ell})}  \widehat{\varphi}((A^\tau)^{-J}\vt-2\pi \vec{\ell}) \overline{\widehat{\varphi}((A^\tau)^{-J}\vt)}\Big) d\vt
\end{equation*}
\end{lemma}

\begin{proof}

By Remarks \ref{vv} after Lemma \ref{properties},
we have $D_A ^J T_{\vec{\ell}}=T_{A^{-J}\vec{\ell}} D_A^{J}.$
Note facts that the Fourier transform $\F$ is a unitary operator, $\overline{\widehat{D}_A^{J}\widehat{\varphi}(\vt)} =
\frac{1}{\sqrt{2^J}}\overline{\widehat{\varphi}((A^\tau)^{-J}\vt)}$ and
$(A^{-J}\vec{\ell})\circ\vt = \vec{\ell}\circ ((A^\tau)^{-J}\vt)$. We have
\begin{eqnarray*}
L_J(f)
&=&\sum_{\vec{\ell} \in \Z^2}  |\langle f,T_{A^{-J}\vec{\ell}} D_A^{J}\varphi \rangle |^2\\
&=&\sum_{\vec{\ell} \in \Z^2} | \langle \widehat{f}, \widehat{T}_{A^{-J}\vec{\ell}} \widehat{D}_A^{J}\widehat{\varphi} \rangle  |^2\\
&=&\sum_{\vec{\ell} \in \Z^2} | \int_{\R^2} \widehat{f}(\vt) \cdot e^{i (A^{-J}\vec{\ell})\circ \vt}
\cdot\overline{\widehat{D}_A^{J}\widehat{\varphi}(\vt)} d\vt \ |^2 \\
&=&\sum_{\vec{\ell} \in \Z^2} | \frac{1}{\sqrt{2^J}}\int_{\R^2} e^{i\vec{\ell}\circ ((A^\tau)^{-J}\vt)}
\cdot\widehat{f}(\vt) \cdot \overline{\widehat{\varphi}((A^\tau)^{-J}\vt)} d\vt \ |^2.
\end{eqnarray*}
Take transform $d \vt \equiv d (A^\tau)^{J} \vs = 2^J d\vs.$ Note that $(A^\tau)^{J}\R^2=\R^2.$
We have
\begin{eqnarray*}
L_J(f)
&=& 2^J \sum_{\vec{\ell} \in \Z^2} \big| \int_{\R^2} e^{i\vec{\ell}\circ \vs}
\cdot\widehat{f}((A^\tau)^{J}\vs) \cdot \overline{\widehat{\varphi}(\vs) } d\vs \ \big|^2.
\end{eqnarray*}
Note facts that the function $e^{i\vec{\ell}\circ \vs}$ is $2\pi$-periodic in $\vs$ and
that the set
$\{\Gamma_\pi +2\pi\vec{k},\vec{k}\in\Z^2\}$ is a partition of $\R^2$ where $\Gamma_\pi$ is $[-\pi,\pi)^2.$ We have
\begin{eqnarray*}
L_J(f)
&=&2^{J}\sum_{\vec{\ell} \in \Z^2} \big|
\sum_{\vec{k} \in \Z^2}
\int_{\Gamma_\pi +2\pi\vec{k}} e^{i\vec{\ell}\circ \vs}\cdot\widehat{f}((A^\tau)^{J}\vs) \cdot \overline{\widehat{\varphi}(\vs)} d\vs \ \big|^2\\
&=& 2^{J}\sum_{\vec{\ell} \in \Z^2} \big|\sum_{\vec{k} \in \Z^2}
\int_{\Gamma_\pi} e^{i\vec{\ell}\circ \vr}
\cdot\big( \widehat{f}((A^\tau)^{J}\vr-2\pi(A^\tau)^{J}\vec{k}) \cdot \overline{\widehat{\varphi}(\vr-2\pi\vec{k})} \big) d\vr \ \big|^2\\
&=& 2^{J}\sum_{\vec{\ell} \in \Z^2} \big|
\int_{\Gamma_\pi} e^{i\vec{\ell}\circ \vr}
\cdot\sum_{\vec{k} \in \Z^2}\big( \widehat{f}((A^\tau)^{J}\vr-2\pi(A^\tau)^{J}\vec{k}) \cdot \overline{\widehat{\varphi}(\vr-2\pi\vec{k})} \big) d\vr \ \big|^2
\end{eqnarray*}
where we use the transform $\vr = \vs+2\pi \vec{k}$ accordingly.

The set of functions $\{\frac{1}{2\pi}e^{i\vec{\ell}\circ \vt},\vec{\ell}\in\Z^2\}$ is an orthonormal basis for the Hilbert space $\mathcal{K} = L^2(\Gamma_\pi),$ the set of all
square integrable $2\pi$-periodical functions on $\R^2.$ Denote
\begin{equation*}
    h(\vt) \equiv \sum_{\vec{k} \in \Z^2}\big( \widehat{f}((A^\tau)^{J}\vt-2\pi(A^\tau)^{J}\vec{k}) \cdot \overline{\widehat{\varphi}(\vt-2\pi\vec{k})} \big).
\end{equation*}
Then by above calculation and Lemma \ref{conv}
\begin{equation*}
\sum_{\vec{\ell} \in \Z^2} \big|
\int_{\Gamma_\pi}h(\vt)
\cdot \frac{1}{2\pi}e^{i\vec{\ell}\circ \vt}d\vt \ \big|^2
=
\frac{1}{2^{J}\cdot (2\pi)^2} \cdot L_J(f)
< \infty.
\end{equation*}
This implies that $h\in \mathcal{K}=L^2(\Gamma_\pi)$ and

\begin{eqnarray*}
  \left\| h \right\|^2 _{\mathcal{K}} &=&
  \sum_{\vec{\ell} \in \Z^2} \big|
\int_{\Gamma_\pi}h(\vt)
\cdot \frac{1}{2\pi}e^{i\vec{\ell}\circ \vt}d\vt \ \big|^2\\
\end{eqnarray*}
where $\|\cdot\|_{\mathcal{K}}$ is the norm in $\mathcal{K}.$
Therefore
\begin{eqnarray*}
L_J(f)
&=&  2^{J}\cdot (2\pi)^2 \cdot \left\| h \right\|^2 _{\mathcal{K}}\\
&=&  2^{J}\cdot (2\pi)^2 \int_{\Gamma_\pi}
\big|\sum_{\vec{k} \in \Z^2}\big( \widehat{f}((A^\tau)^{J}\vt-2\pi(A^\tau)^{J}\vec{k}) \cdot \overline{\widehat{\varphi}(\vt-2\pi\vec{k})} \big)\big|^2 d \vt\\
&=& (2\pi)^2
\int_{(A^\tau)^J\Gamma_\pi} \big|\sum_{\vec{k} \in \Z^2}\widehat{f}(\vs-2\pi(A^\tau)^{J}\vec{k}) \cdot \overline{\widehat{\varphi}((A^\tau)^{-J}\vs-2\pi\vec{k})} \big|^2 \ d\vs.
\end{eqnarray*}
Here we use a transform $\vt \equiv (A^\tau) ^{-J} \vs,\ d\vt =2^{-J} d\vs.$
So we have
\begin{eqnarray*}
L_J(f)
&=& (2\pi)^2
\int_{(A^\tau)^J\Gamma_\pi}\sum_{\vec{k} \in \Z^2}\sum_{\vec{\ell} \in \Z^2}\Big(\widehat{f}(\vs-2\pi(A^\tau)^{J}\vec{k}) \cdot
 \overline{\widehat{f}(\vs-2\pi(A^\tau)^{J}\vec{\ell})} \cdot
\\
&&
\widehat{\varphi}((A^\tau)^{-J}\vs-2\pi\vec{\ell})\cdot
\overline{\widehat{\varphi}((A^\tau)^{-J}\vs-2\pi\vec{k})}
\Big) d\vs.
\end{eqnarray*}
In the second sum, replace $\vec{\ell}$ by $\vec{\ell}+\vec{k},$ we have
\begin{eqnarray*}
L_J(f)
&=& (2\pi)^2
\sum_{\vec{k} \in \Z^2}\int_{(A^\tau)^J\Gamma_\pi}\sum_{\vec{\ell} \in \Z^2}\Big(\widehat{f}(\vs-2\pi(A^\tau)^{J}\vec{k}) \cdot
\overline{\widehat{f}(\vs-2\pi(A^\tau)^{J}\vec{k}-2\pi(A^\tau)^{J}\vec{\ell})} \cdot\\
&&
\widehat{\varphi}((A^\tau)^{-J}\vs-2\pi\vec{k}-2\pi\vec{\ell})\cdot
\overline{\widehat{\varphi}((A^\tau)^{-J}\vs-2\pi\vec{k})}
\Big) d\vs.
\end{eqnarray*}
Replacing $\vs$ by $\vs+2\pi(A^\tau)^{J}\vec{k},$ we have
\begin{eqnarray*}
L_J(f)
&=& (2\pi)^2
\sum_{\vec{k} \in \Z^2}\int_{(A^\tau)^J\Gamma_\pi+2\pi (A^\tau)^J \vec{k} } \sum_{\vec{\ell} \in \Z^2}\Big(\widehat{f}(\vs) \cdot
\overline{\widehat{f}(\vs-2\pi(A^\tau)^{J}\vec{\ell})} \cdot\\
&&
\widehat{\varphi}((A^\tau)^{-J}\vs-2\pi\vec{\ell})\cdot
\overline{\widehat{\varphi}((A^\tau)^{-J}\vs)}
\Big) d\vs.
\end{eqnarray*}
Since $\{(A^\tau)^J\Gamma_\pi+2\pi (A^\tau)^J \vec{k}, \vec{k}\in\Z^2 \}$
is a partition of $\R^2,$ we have
\begin{eqnarray*}
L_J(f)
&=& (2\pi)^2
\int_{\R^2}\sum_{\vec{\ell} \in \Z^2}\Big(\widehat{f}(\vs) \cdot
\overline{\widehat{f}(\vs-2\pi(A^\tau)^{J}\vec{\ell}) }\cdot
\widehat{\varphi}((A^\tau)^{-J}\vs-2\pi\vec{\ell})\cdot
\overline{\widehat{\varphi}((A^\tau)^{-J}\vs)}
\Big) d\vs.
\end{eqnarray*}
The Lemma is proved.

\end{proof}

\begin{proposition}\label{isnorm} We have
\begin{equation}
\lim _{J\rightarrow + \infty} L_J(f) = \|f\|^2, \forall f\in L^2(\R^2).
\end{equation}
\end{proposition}

\begin{proof}
We denote
\begin{eqnarray*}
  U_J (f) &\equiv&
  (2\pi)^2
\int_{\R^2}\widehat{f}(\vs) \cdot
\overline{\widehat{f}(\vs)} \cdot
\widehat{\varphi}((A^\tau)^{-J}\vs)\cdot
\overline{\widehat{\varphi}((A^\tau)^{-J}\vs)}d\vs\\
&=&\int_{\R^2} |\widehat{f}(\vs)|^2 |2\pi\widehat{\varphi}((A^\tau)^{-J}\vs) |^2 d\vs,\\
  V_J (f) &\equiv&
  (2\pi)^2
\int_{\R^2}\sum_{\vec{\ell} \in \Z^2\backslash\{\vec{0}\}}\Big(\widehat{f}(\vs) \cdot
\overline{\widehat{f}(\vs-2\pi(A^\tau)^{J}\vec{\ell})} \cdot
\widehat{\varphi}((A^\tau)^{-J}\vs-2\pi\vec{\ell})\cdot
\overline{\widehat{\varphi}((A^\tau)^{-J}\vs)}
\Big) d\vs.
\end{eqnarray*}
By Lemma \ref{identitylj}, we have
$L_J(f) = U_J (f) + V_J (f).$ It is enough to prove that
\begin{eqnarray}
\label{r}  \lim_{J\rightarrow +\infty} U_J (f) &=& \| f \| ^2 \text{ and}\\
\label{s}  \lim_{J\rightarrow +\infty} V_J (f) &=& 0.
\end{eqnarray}

1. Recall that $A^\tau$ is expansive, so $\lim_{J\rightarrow +\infty}(A^\tau)^{-J}\vs=\vec{0}, \forall \vs\in\R^2.$
Also, by definition of $g$, Remark \ref{less1} and Lemma \ref{entire}, $2\pi \widehat{\varphi} (\vec{0}) = 2\pi g(\vec{0})=1,$ $g$ is continuous and bounded on $\R^2$.  By Lebesgue Dominate Convergence Theorem we have
\begin{eqnarray*}
  \lim_{J\rightarrow +\infty} U_J (f) &=&
  \lim_{J\rightarrow +\infty}
  \int_{\R^2}|\widehat{f}(\vs)|^2 \cdot
|2\pi \widehat{\varphi}((A^\tau)^{-J}\vs)|^2
d\vs\\
&=&   \lim_{J\rightarrow +\infty}
  \int_{\R^2}|\widehat{f}(\vs)|^2 \cdot
|2\pi g((A^\tau)^{-J}\vs)|^2
d\vs\\
&=&  \|\widehat{f}\|^2
=  \|f\|^2.
\end{eqnarray*}
This proves \eqref{r}.

2. Let $\rho\in\R^+,$ $\Delta_\rho$ be the open ball with center $\vec{0}$ and radius $\rho.$ In particular,
$\Delta_1$ is the open ball with center $\vec{0}$ and radius $1.$
Let $\chi_\rho$ and  $\chi_{\overline{\rho}}$ be the characteristic functions of sets $\Delta_\rho$ and $\R^2\backslash \Delta_\rho$ respectively. Define $f_\rho$ by
$ \widehat{f_\rho} \equiv \chi_\rho \widehat{f}$ and  define $f_{\overline{\rho}}$ by
$ \widehat{f_{\overline{\rho}}} \equiv \chi_{\overline{\rho}} \widehat{f}.$
Since the Fourier transform is linear, we have $f=f_\rho+f_{\overline{\rho}}.$ Also, it is clear that $\|f\|^2 = \|\widehat{f}\|^2=\|f_\rho\|^2 +\|f_{\overline{\rho}}\|^2,$
 $\lim_{\rho\rightarrow+\infty}\|f_\rho\|^2=\|f\|^2$ and
$\lim_{\rho\rightarrow+\infty}\|f_{\overline{\rho}}\|^2=0.$

Since $A^\tau$ is expansive, $\beta\equiv \|(A^\tau)^{-1}\|^{-1}>1.$ Denote $a\equiv \log_\beta(2\rho).$ Let  $J_{\rho}$ be the smallest natural number in the interval
$(a,+\infty).$
 When $J\geq J_\rho,$  $(A^\tau)^J \Delta_1$ contains an open ball $\Delta_{2\rho}.$
 Since $\Delta_1\cap \Z^2 =\{\vec{0}\},$
$2\pi (A^\tau)^J\Delta_1\cap 2\pi (A^\tau)^J\Z^2 =\{\vec{0}\}.$ Also we have $\Delta_{2\rho} \subseteq (A^\tau)^J\Delta_1\subseteq 2\pi (A^\tau)^J\Delta_1.$
These facts implies that when $J\geq J_\rho$ the distance between $\vec{0}$ and $2\pi (A^\tau)^J \vec{\ell}$ is greater than $2\rho.$
So for each $\vec{\ell}\in\Z^2\backslash\{\vec{0}\}$, the support of $\widehat{f_\rho}(\vt)$ which is $\Delta_\rho$ and the  support of $\widehat{f_\rho}(\vt-2\pi (A^\tau)^J\vec{\ell})$
which is $\Delta_\rho+2\pi (A^\tau)^J\vec{\ell}$ are disjoint.
This implies that the product
$\widehat{f_\rho}(\vt)\overline{\widehat{f_\rho}(\vt-2\pi (A^\tau)^J\vec{\ell})}\equiv 0$
when $J\geq J_\rho.$ Therefore, we have
\begin{equation*}
    \lim_{J\rightarrow+\infty} V_J (f_\rho) = 0, \forall \rho\in\R^+.
\end{equation*}
Together with \eqref{r}, we have proved that
\begin{equation}\label{i}
\lim _{J\rightarrow + \infty} L_J(f_\rho) = \|f_\rho\|^2,
\forall f\in L^2(\R^2),\forall\rho\in\R^+.
\end{equation}

3. Let $D_\rho \equiv \sum_{\vec{\ell}\in\Z^2} \big(\langle f_\rho,D_A ^J T_{\vec{\ell}}\varphi \rangle
  \overline{\langle f_{\overline{\rho}} ,D_A ^J T_{\vec{\ell}}\varphi \rangle}+
  \langle f_{\overline{\rho}} ,D_A ^J T_{\vec{\ell}}\varphi \rangle
  \overline{\langle f_\rho ,D_A ^J T_{\vec{\ell}}\varphi \rangle}\big).$
  Then
  \begin{eqnarray*}
   | D_\rho | &\leq& 2\sum_{\vec{\ell}\in\Z^2} | \langle f_\rho,D_A ^J T_{\vec{\ell}}\varphi \rangle| \cdot
  |\langle f_{\overline{\rho}} ,D_A ^J T_{\vec{\ell}}\varphi \rangle| \\
  &\leq&
  2\sqrt{\sum_{\vec{\ell}\in\Z^2} | \langle f_\rho,D_A ^J T_{\vec{\ell}}\varphi \rangle|^2} \cdot
  \sqrt{\sum_{\vec{\ell}\in\Z^2}|\langle f_{\overline{\rho}} ,D_A ^J T_{\vec{\ell}}\varphi \rangle|^2}\\
  &=&
  2 \sqrt{L_J(f_\rho)} \cdot \sqrt{L_J(f_{\overline{\rho}})}.
  \end{eqnarray*}
By Lemma \ref{conv}, we have
\begin{eqnarray*}
 | D_\rho | &\leq&  2 (2B+1)^2 \|\varphi\|^2 \|f_\rho\| \|f_{\overline{\rho}}\|.
\end{eqnarray*}
We have
\begin{eqnarray*}
  L_J(f)-\|f\|^2
  &=& L_J(f_\rho + f_{\overline{\rho}}) -\|f\|^2 \\
  &=& \sum_{\vec{\ell}\in\Z^2} \langle f_\rho + f_{\overline{\rho}} ,D_A ^J T_{\vec{\ell}}\varphi \rangle
  \overline{\langle f_\rho + f_{\overline{\rho}} ,D_A ^J T_{\vec{\ell}}\varphi \rangle}-\|f\|^2\\
  &=& L_J(f_\rho )-\|f\|^2 + L_J(f_{\overline{\rho}})+D_\rho\\
  &=& \big(L_J(f_\rho )-\|f_\rho\|^2\big) -\|f_{\overline{\rho}}\|^2 + L_J(f_{\overline{\rho}})+D_\rho.
\end{eqnarray*}
By \eqref{i}, \eqref{y} and Lemma \eqref{x} we have
  \begin{eqnarray*}
&&\limsup_{J\rightarrow + \infty}\big| L_J(f)-\|f\|^2\big|\leq 0+ \|f_{\overline{\rho}}\|^2 + 0 +
2 (2B+1)^2 \|\varphi\|^2 \|f_\rho\| \|f_{\overline{\rho}}\|, \forall \rho\in\R^+.
\end{eqnarray*}
The left hand side contains no $\rho.$ Let $\rho\rightarrow+\infty,$
we have \eqref{s}
\begin{eqnarray*}
  \lim_{J\rightarrow \infty} L_J (f)  &=& \|f\|^2.
\end{eqnarray*}
\end{proof}

The proof of Theorem \ref{theom_frame} is complete.

\bigskip

\section{Conclusion}

Let $A_0$ be a $2\times 2$ expansive integral matrix with $|\det (A_0)|=2.$ We can construct normalized tight frame wavelets associated with $A_0$ in the following steps.

(1) Find a $2\times 2$ integral matrix $S$ with $|\det (S)|=1$ with the property that
$S A S^{-1} = A_0$ where $A$ is one of the six matrices in list  (\ref{newsix}) (Proposition \ref{linktonew6}).

(2) Solve the system of equations (\ref{lawtoneq})
\begin{equation*}
\left\{\begin{array}{l}
\sum_{\vec{n}\in\Z^2}h_{\vec{n}}\overline{h_{\vec{n}+\vec{k}}}=\delta_{\vec{0} \vec{k}},~ \vec{k}\in A^\tau\Z^2 \\
\sum_{\vec{n}\in\Z^2}h_{\vec{n}}=\sqrt{2}.
\end{array}\right.
\end{equation*}
for a finite solution $\mathcal{S}=\{h_{\vec{n}}:~{\vec{n}\in\Z^2}\}.$ That is, the index set of non-zero terms $h_{\vec{n}}$ is included in
the set $\Lambda_0 \equiv \Z^2 \cap [-N_0,N_0]^2$ for some natural number $N_0.$

(3) Define the filter function $m_0$ by (\ref{m0})
\begin{equation*}
    m_0(\vec{t})=\frac{1}{\sqrt{2}}\sum_{\vec{n}\in\Z^2} h_{\vec{n}} e^{-i\vec{n}\circ\vec{t}}.
 \end{equation*}

(4) Define a function $g$ by (\ref{g}),
\begin{equation*}
 g(\vec{\xi})= \frac{1}{2\pi}\prod_{j=1}^{\infty}m_0((A^{\tau})^{-j} \vec{\xi}), \forall \vec{\xi}\in\R^2
\end{equation*}
The function $g$ is an $L^2(\R^2)$-function (Proposition \ref{l2phi}).

(5) Define the scaling function $\varphi$ by (\ref{phi})
\begin{equation*}
\varphi = \F^{-1} g.
\end{equation*}
The scaling function $\varphi$ is an $L^2(\R^2)$-function with compact support (Proposition \ref{compactsupport}).

(6) Let $\vec{\ell}_A$ be the vector as in Proposition \ref{properties}. Define
\begin{equation*}
    \sigma_A (\vec{n}) = \left\{
    \begin{array}{cccc}
      0 & \vec{n} \in A \Z^2,  \\
      1 & \vec{n} \notin A \Z^2.
    \end{array}
    \right.
\end{equation*}
Define the wavelet function  $\psi_A$ on $\R^2$ by (\ref{defpsi})
\begin{eqnarray*}
  \psi_A &=& \sum_{\vec{n}\in\Z^2}
  (-1)^{\sigma_A (\vec{n})} \overline{h_{\vec{\ell}_A-\vec{n}}} D_AT_{\vec{n}} \varphi .
\end{eqnarray*}
This is a normalized tight frame wavelet with compact support associated with matrix $A$ (Theorem \ref{theom_frame}).

(7) Define the wavelet function $\psi$ by
\begin{equation*}
    \psi(\vec{t}) \equiv \psi_A(S\vec{t}),\forall \vec{t}\in\R^2.
\end{equation*}
The function $\psi$ is a normalized tight frame wavelet with compact support associated with the given matrix $A_0$ (Theorem \ref{redsym}).

\end{document}